\documentclass{amsart}

\usepackage[utf8]{inputenc}
\usepackage{color}
\usepackage{cite}
\usepackage{hyperref}
\usepackage{amsmath}
\usepackage{amsthm}
\usepackage{amsfonts}
\usepackage{amssymb}
\usepackage[english]{babel}
\usepackage{lmodern}
\usepackage{graphicx}

\usepackage{geometry}
\hypersetup{linktocpage,colorlinks, citecolor={magenta}}
\numberwithin{equation}{section}

\usepackage{enumitem}
\setlist{nosep}
\setlist{noitemsep}

\theoremstyle{definition}
\newtheorem{theorem}{Theorem}
\newtheorem{prop}{Proposition}[section]
\newtheorem{lemma}[prop]{Lemma}
\newtheorem{coro}[prop]{Corollary}

\newtheorem{claim}[prop]{Claim}

\newtheorem{remark}[prop]{Remark}
\newtheorem{definition}[prop]{Definition}

\newcommand{\fbeta}{\mathcal{F}_{\beta}}
\newcommand{\Welec}{\mathcal{W}}
\newcommand{\Wint}{\mathcal{W}^{\mathrm{int}}}
\newcommand{\Ent}{\mathrm{Ent}}
\newcommand{\SRE}{\mathcal{E}} 

\newcommand{\config}{\mathrm{Conf}}

\newcommand{\C}{\mathcal{C}}
\newcommand{\Int}{\mathrm{Int}}

\newcommand{\sinebeta}{\mathrm{Sine}_{\beta}}

\renewcommand{\P}{\mathsf{P}}

\newcommand{\hP}{\widehat{\mathsf{P}}}

\newcommand{\OrN}{\mathcal{O}_R}
\newcommand{\piN}{\pi_R}
\newcommand{\La}{\Lambda}
\newcommand{\configNN}{\config_R}

\def \E{\mathbb{E}}
\def \1{\mathtt{1}}
\def \R{\mathbb{R}}
\def \id{\mathrm{Id}}
\def \z{\mathsf{0}}
\def \u{\mathsf{1}}
\def \X{\mathrm{X}}
\def \h{\frac{1}{2}}
\def \Gap{\Gamma}
\def \C{\mathcal{C}}

\def \Z{\mathbb{Z}}

\def \Lploc{L^{p}_{loc}}

\def \div{\mathrm{div}}

\def \c{2\pi}

\def \Screen{\mathsf{Screen}}
\def \f{\mathsf{f}}

\def \scr{\mathrm{scr}}
\def \Cscr{\C^{\scr}}
\def \Escr{\mathrm{E}^{\scr}}
\def \Phiscr{\Phi^{\scr}}
\def \ErrScr{\mathrm{ErrScr}}
\def \Old{\mathsf{Old}}
\def \New{\mathsf{New}}

\def \Discr{\mathrm{Discr}}
\def \xx{\mathrm{x}}
\def \Xz{\X^{\z}}
\def \Xu{\X^{\u}}

\def \Cz{\C^{\z}}
\def \Cu{\C^{\u}}
\def \Ch{\C^{\h}}

\def \dd{\mathsf{d}}
\def \ss{\mathsf{s}}
\def \kk{\mathsf{k}}
\def \l{\ell}
\def \HRT{H_{R,T}}
\def \tail{\mathrm{Tail}}

\def \Nint{N_{\mathrm{int}}}
\def \UZ{U_0}
\def \tH{\tilde{H}}

\def \Mec{\mathsf{M}_{\mathrm{scr}}}
\def \Hec{\mathsf{e}_{\mathrm{scr}}}
\def \epsilon{\varepsilon}

\newcommand{\hal}{\frac{1}{2}}

\def \EventEnergy{\mathsf{GEnergy}}  
\def \EventDecay{\mathsf{GDecay}}
\def \EventDiscr{\mathsf{GDiscr}}

\def \EventTrunc{\mathsf{GTrunc}}

\def \llbr{\left\lbrace}
\def \rrbr{\right\rbrace}

\def \EveR{\mathsf{GEvent}_R}
\def \tPR{\widetilde{\P}_R}

\def \etac{\eta_{\mathrm{scr}}}

\def \Emp{\mathrm{Emp}}
\def \XN{\vec{X}_N}

\def \PNbeta{\mathbb{P}_{N, \beta}}
\def \ZNbeta{\mathrm{Z}_{N, \beta}}

\usepackage{mathtools}
\DeclarePairedDelimiter\ceil{\lceil}{\rceil}
\DeclarePairedDelimiter\floor{\lfloor}{\rfloor}

\def \Wel{\mathrm{W}^{\mathrm{elec}}}
\def \tWel{\widetilde{\mathrm{W}}^{\mathrm{elec}}}
\def \Wint{\mathrm{W}^{\mathrm{int}}}
\def \ErMon{\mathrm{Er}\mathrm{Mono}}
\def \ErTrun{\mathrm{Er}\mathrm{Trun}}

\def \mdiag{(\La_R \times \La_R) \setminus \diamond}
\def \hh{\mathsf{h}}
\def \Phal{\P^{\hh}}
\def \Chal{\C^{\hh}}

\def \Cscrz{\C^{\z, \scr}}
\def \Cscru{\C^{\u, \scr}}
\def \El{\mathrm{E}}
\def \Eloc{\El^{\mathrm{loc}}}

\def \hGap{\mathrm{G}}
\def \hHR{\check{H}_R}
\def \Pos{\mathrm{Pos}}
\def \EveGap{\mathsf{EGap}}
\def \Error{\mathsf{Error}}
\def \Kmax{k_{\mathrm{max}}}

\def \Rdix{\frac{R}{10}}

\def \CLr{\mathsf{C}_{L,r}}

\def \Pz{\P^{\z}}
\def \Pu{\P^{\u}}
\def \Id{\mathrm{Id}}
\def \V{\mathsf{V}}
\def \lp{\left(}
\def \rp{\right)}
\def \const{\mathrm{const}}

\def \hPz{\hP^{\z}}
\def \hPu{\hP^{\u}}
\def \TN{\mathsf{T}_{R}}
\def \Xhal{\X^{\hh}}
\def \Chal{\C^{\hh}}
\def \GGain{\mathsf{Gain}}
\def \BF{\mathsf{BF}}
\def \g{\mathfrak{g}}
\def \kkmax{k_{\mathrm{max}}}

\newcommand{\BN}{\mathsf{B}_R}
\newcommand{\LN}{\mathbf{L}_R}
\def \Hleft{H_{\mathrm{left}}}
\def \Hright{H_{\mathrm{right}}}

\def \Trans{\mathsf{T}}
\def \NoPoint{\mathsf{GD}_1}
\def \LeftRight{\mathsf{GD}_2}
\def \chP{\check{\P}}
\def \Wintg{\mathcal{W}^{\mathrm{int}}}
\def \TotDiscr{\mathsf{GD_3}}
\def \Xx{\mathsf{z}}
\def \bXx{\overline{\Xx}}

\def \Ptot{\P^{\mathrm{tot}}}
\def \Ptotz{\P^{\mathrm{tot},z}}
\def \Pav{\P^{\mathrm{av}}}
\def \Paste{\mathrm{Paste}}
\def \Cleft{\C_{\mathrm{left}}}
\def \Cright{\C_{\mathrm{right}}}
\def \tD{\widetilde{D}^{\mathrm{Left}}}
\def \Qo{\mathsf{Q}}
\def \Q{\mathsf{Q}}
\newcommand{\Poisson}{\mathsf{\Pi}}
\def \tPzR{\widetilde{\P}^{\z}_R}
\def \tPuR{\widetilde{\P}^{\u}_R}
\def \tQo{\widetilde{\Qo}}
\def \GoodField{\mathsf{GField}}
\def \Interaction{\mathrm{Interaction}}

\begin{document}
\begin{abstract}
We prove that, at every positive temperature, the infinite-volume free energy of the one dimensional log-gas, or beta-ensemble, has a unique minimiser, which is the Sine-beta process arising from random matrix theory. We rely on a quantitative displacement convexity argument at the level of point processes, and on the screening procedure introduced by Sandier-Serfaty.
\end{abstract}
\title{The one-dimensional log-gas free energy has a unique minimiser}
\author{Matthias Erbar, Martin Huesmann, Thomas Leblé}
\address[Matthias Erbar]{Institut f\"ur angewandte Mathematik, Universit\"at Bonn, Endenicher Allee 60, 53115 Bonn, Germany}\email{erbar@iam.uni-bonn.de}
\address[Martin Huesmann]{Universit\"at Wien,  Oskar-Morgenstern-Platz 1, 1090 Wien, Austria} \email{martin.huesmann@univie.ac.at}
\address[Thomas Leblé]{Courant Institute of Mathematical Sciences, 251 Mercer Street, New York University, New York, NY 10012-1110, USA} \email{thomasl@math.nyu.edu}

\thanks{ME gratefully acknowledges support by the German Research
  Foundation through the Hausdorff Center for Mathematics and the
  Collaborative Research Center 1060 \emph{Mathematics of Emergent
    Effects}. MH has been funded by the Vienna Science and Technology
  Fund (WWTF) through project VRG17-005.}

\date{\today}
\maketitle

\section{Introduction and main result}
\subsection{The one-dimensional log-gas}
The one-dimensional log-gas in finite volume can be defined as a system of particles interacting through a repulsive pairwise potential proportional to the logarithm of the distance, and confined by some external field.

For a fixed value of $\beta > 0$, called the \textit{inverse temperature} parameter, and for $N \geq 1$, we consider the probability measure $\PNbeta$ on $\XN = (x_1, \dots, x_N) \in \R^N$ defined by the density
\begin{equation}
\label{def:PNbeta}
d\PNbeta(\XN) := \frac{1}{\ZNbeta} \exp\left( - \beta \left( \sum_{i < j} - \log |x_i - x_j| + \sum_{i=1}^N N \frac{x_i^2}{2} \right) \right), 
\end{equation}
with respect to the Lebesgue measure on $\R^N$. The quantity $\ZNbeta$ is a normalization constant, the \textit{partition function}. We call $\PNbeta$ the \textit{canonical Gibbs measure} of the log-gas. 

Part of the motivation for studying log-gases comes from Random Matrix Theory (RMT), for which $\PNbeta$ describes the joint law of $N$ eigenvalues in certain classical models: the Gaussian orthogonal, unitary, symplectic ensemble respectively for $\beta = 1,2,4$, and the “tridiagonal model” discovered in \cite{dumitriu2002matrix} for arbitrary $\beta$. We refer to the book \cite{forrester2010log} for a comprehensive presentation of the connection between log-gases and random matrices. Log-gases are also interesting from a statistical physics point of view, as a toy model with singular, long-range interaction.

Questions about such systems usually deal with the large $N$ limit (also called thermodynamic, or infinite-volume limit) of the system, as encoded by certain observables. For example, in order to understand the “global” behavior, one may look at the empirical measure $\frac{1}{N} \sum_{i=1}^N \delta_{x_i}$,
and asks about the typical behavior of this \textit{random} probability measure on $\R$ as $N$ tends to infinity. By now, this is fairly well understood, we refer e.g. to the recent lecture notes \cite{serfaty2017microscopic} and the references therein. In the present paper, we are rather interested in the asymptotic behavior at \textit{microscopic} scale.

\subsection{The Sine-beta process}
Let $\C_{N,0}$ be the point configuration 
$$
\C_{N, 0} := \sum_{i=1}^N \delta_{N x_i},
$$ 
where $\XN  = (x_1, \dots, x_N)$ is distributed according to $\PNbeta$. The limit in law of $\C_{N,0}$ as $N \to \infty$ was constructed in  \cite{valko2009continuum} and named the $\sinebeta$ process. We refer to \cite{killip2009eigenvalue} for a different construction of a process that turns out to be same, and to \cite{valko2017sine, valko2017operator} for recent developments concerning $\sinebeta$. This process is the universal behavior of log-gases (in the bulk), in the sense that replacing the $\frac{x_i^2}{2}$ term in \eqref{def:PNbeta} by a general potential $V(x_i)$ yields the same microscopic limit, up to a scaling on the average density of points (our convention is that $\sinebeta$ has intensity $1$) and mild assumptions on $V$, see \cite{Bourgade1d1, Bourgade1d2}.

In \cite{dereudre2018dlr}, a different description of $\sinebeta$ is given using the Dobrushin-Landford-Ruelle (DLR) formalism, but the question of whether $\sinebeta$ is the unique solution to DLR equations is left open. The main result of the present paper answers positively to a slightly different uniqueness question, phrased in terms of the log-gas free energy.

\subsection{The log-gas free energy}
In \cite{leble2017large}, the infinite-volume free energy of the log-gas (and of other related systems) was introduced as the weighted sum $\fbeta := \beta \Welec + \SRE$, where the functionals $\Welec, \SRE$ and the free energy $\fbeta$ are defined on the space of stationary random point processes. The functional $\Welec$ corresponds to the “renormalized energy” introduced in \cite{sandier2012ginzburg}, and adapted to this context in \cite{sandier20151d, petrache2017next}, and $\SRE$ is the usual \textit{specific relative entropy}. Both terms are defined below, see Section \ref{sec:defoffree}.

The free energy $\fbeta$ appears in \cite{leble2017large} as the rate function for a large deviation principle concerning the behavior of log-gases at the microscopic level. If $\XN = (x_1, \dots, x_N)$ is an $N$-tuple of particles distributed according to the Gibbs measure \eqref{def:PNbeta} of a log-gas, they are known to typically arrange themselves on an interval approximately given by $[-2, 2]$. For $\xx$ in this interval, we let $\C_{N, \xx}$ be the point configuration $(x_1, \dots, x_N)$ “seen from $\xx$”, namely $
\C_{N, \xx} := \sum_{i=1}^N \delta_{N(x_i - \xx)}$.
We may then consider the \textit{empirical field} of the system in the state $\XN$, defined as
$$
\Emp_N(\XN) := \frac{1}{4} 
\int_{-2}^{2} \delta_{\C_{N, \xx}} d \xx.
$$
The empirical field $\Emp_N(\XN)$ is a probability measure on (finite) point configurations in $\R$, and it was proven in \cite{leble2017large} that its law satisfies a large deviation principle, at speed $N$, with a rate function built using $\fbeta$. We refer to the paper cited above for a precise statement, here it suffices to say that \textit{understanding the minimisers} of $\fbeta$ gives an \textit{understanding of the typical microscopic behavior} of a finite $N$ log-gas at temperature $\beta$, when $N$ is large.

For any $\beta$ in $(0, +\infty)$, the functional $\fbeta$ is known to be lower semi-continuous, with compact sub-level sets. In particular, it admits a compact subset of minimisers. However, the question of uniqueness of minimisers for $\fbeta$ remained open, and we address it in this paper.

\subsection{Main result}
\begin{theorem}
\label{theo:main}
For any $\beta$ in $(0, + \infty)$, the free energy $\fbeta$ has a unique minimiser. 
\end{theorem}

Since it was proven in \cite[Corollary 1.2]{leble2017large} that $\sinebeta$ minimises $\fbeta$, we deduce that:
\begin{coro} 
\label{coro:sinebeta}
For any $\beta$ in $(0, + \infty)$, the $\sinebeta$ process is the unique minimiser of $\fbeta$.
\end{coro}
This provides a variational characterization of $\sinebeta$. A weaker variational property of $\sinebeta$ had been used e.g. in \cite{leble2016logarithmic} to prove the convergence of $\sinebeta$ to a Poisson point process as $\beta \to 0$, retrieving a result of \cite{allez2014sine}.

\subsection{Elements of proof and plan of the paper}
The proof of Theorem \ref{theo:main} goes by contradiction. We assume that $\fbeta$ admits two distinct minimisers, and we construct a stationary point process whose free energy is stricly less than $\min \fbeta$, which is of course absurd. Since the free energy $\fbeta$ is affine, it is not strictly convex for the usual linear interpolation of probability measures. We use instead the notion of \textit{displacement convexity}, which was introduced in \cite{MCCANN1997153} to remedy situations where energy functionals are not convex in the usual sense. 
This idea was suggested to us by Alice Guionnet, and we warmly thank her for her insight.

\subsubsection*{Strategy of the proof}
The proof goes by contradiction. We start with two stationary point processes $\Pz, \Pu$ such that $\Pz \neq \Pu$, and assume that both are minimisers of $\fbeta$. We cannot argue via displacement convexity directly on the level of $\Pz, \Pu$ since they are probability measures on \textit{infinite} point configurations. Optimal transport theory for random stationary measures as initiated in \cite{HuSt13, Hu16, ErHu15} is not yet developed well enough to be directly applicable. Instead, we use transport theory between \textit{finite} measures together with a careful approximation argument relying on screening of electric fields. More precisely, we write
$$
\fbeta(\P) = \lim_{R \to \infty}  \frac{1}{|\La_R|} \left(\beta \Welec_R(\P) + \SRE_R(\P) \right),
$$
where $\Welec_R, \SRE_R$ are quantities (the energy, and the relative entropy) depending on the restriction of $\P$ to the line segment $\La_R := [-R, R]$.
\begin{enumerate}
\item Let $\Pz_R, \Pu_R$ be the restriction of $\Pz, \Pu$ to $\La_R$. Assume that there are almost surely $2R$ points in $[-R,R]$. We may thus see $\Pz_R, \Pu_R$ as probability measures on $[-R, R]^{2R}$, apply classical results about optimal transportation, and find an optimal transport map $\Trans_R$ which pushes $\Pz_R$ onto $\Pu_R$.

In fact, it is not true that $\Pz, \Pu$ have almost surely $2R$ points in $\La_R$, so we first perform a version of the “screening procedure” of Sandier-Serfaty. It has for effect to modify the configurations near the extremities of $[-R, R]$ in order to enforce the correct number of points, while not changing too much the average energy, nor the entropy. The screening procedure requires some conditions in order to be applied, we will need to guarantee that they are verified with high probability under $\Pz, \Pu$. 

\item We let $\Phal_R$ be the half-interpolate of $\Pz_R, \Pu_R$ along the displacement $\Trans_R$, i.e. the push-forward of $\Pz_R$ by $\hal (\Id + \Trans_R)$. A result of \cite{MCCANN1997153} ensures that the relative entropy is displacement convex, so 
$$
\SRE_R[\Phal_R] \leq \hal \left( \SRE_R(\Pz) + \SRE_R(\Pu)\right).
$$
Moreover, the interaction potential $-\log |x-y|$ is strictly convex, so again by a result of \cite{MCCANN1997153}, the energy $\Welec_R$ is also displacement convex. More precisely, we have
$$
\Welec_R[\Phal_R] \leq \hal \left( \Welec_R(\Pz) + \Welec_R(\Pu)\right) - \mathrm{Gain}_R,
$$
where $\mathrm{Gain}_R>0$ is some quantitative positive gain due to the \textit{strict} convexity of the interaction. With some work, using the fact that $\Pz, \Pu$ are stationary, we are able to show that the gain is at least proportional to $R$.

\item We turn $\Phal_R$ into a process on the full line by pasting independent copies of itself on disjoint intervals of length $2R$. The relative entropy is additive, and we can show that the interaction of two independent copies is almost zero. Thanks to the quantitative convexity estimate, we obtain a global candidate $\Phal$ for which
$$
\fbeta(\Phal) < \hal \left( \fbeta(\Pz) + \fbeta(\Pu)\right),
$$
which is absurd, hence the minimiser of $\fbeta$ is unique.
\end{enumerate}

\subsubsection*{Plan of the paper}
\begin{itemize}
\item In Section \ref{sec:defoffree}, we define formally the logarithmic energy and the specific relative entropy, which are the two ingredients of the log-gas free energy.
\item In Section \ref{sec:descriscri}, we present the screening procedure.
\item In Section \ref{sec:discregaps}, we state a key discrepancy estimate, and collect auxiliary results controlling the typical size of the gaps, and positions of points, for processes of finite logarithmic energy.
\item In Section \ref{sec:largeboxscreening}, we implement the first step of the proof, by performing the screening procedure on a large box for any given point process with finite energy.
\item In Section \ref{sec:displacement}, we recall results from optimal transportation and the theory of displacement convexity.
\item In Section \ref{sec:interpolateprocess}, we apply displacement interpolation to build a local candidate: the interpolate process in a large box, whose local energy and entropy are better than the two processes it was constructed from.
\item In Section \ref{sec:conclusionproof}, we use the interpolate process to build a global candidate, that is a stationary point process whose free energy is strictly smaller than the minimum of $\fbeta$, which is the desired contradiction.
\item Finally, in Appendix \ref{sec:annex}, we gather the proofs of some technical lemmas, and in Appendix \ref{sec:proofscreening} we sketch a proof of the screening procedure adapted to the present setting.
\end{itemize}

\subsection{Preliminary definitions and notation}
\subsubsection*{Notation}
\begin{itemize}
\item For $R > 0$, we denote by $\La_R$ the interval $[-R,R]$.
\item If $I$ is an interval of $\R$, we let $|I|$ be its length.
\item If $A$ is a subset of $\R$ or $\R^2$, we denote the boundary of $A$ by $\partial A$. 
\item We let $\diamond$ be the diagonal in $\R \times \R$.
\item We will use the symbol $\preceq$ to denote an inequality that holds up to some universal (positive) multiplicative constant. 
\item Points in $\R \times \R$ are denoted by upper case letters, e.g. $X = (x,y)$.
\item If $\mu$ is a signed measure on some interval $I \subset \R$, we will sometimes treat it as a signed measure on $I \times \R$, where it is understood that, $\varphi : I \times \R \to \R$ being a test function, we define
\begin{equation}
\label{conventionmeasure}
\int_{I \times \R} \varphi(x,y) d\mu(x,y) := \int_{I} \varphi(x,0) d\mu(x).
\end{equation}
\item If $\mu$ is a probability measure on $X$ and $F$ is a map from $X$ to $Y$, we let $F_* \mu$ be the push-forward of $\mu$ by $F$, which is a probability measure on $Y$.
\end{itemize}

\subsubsection*{Point configurations and processes}
\begin{itemize}
\item A point configuration $\C$ on a set $S\subset \R$ is
    defined as a purely atomic Radon measure on $S$, giving integer
    mass to singletons and finite mass to any compact set. Any
    configuration can be written as $\C = \sum_{i\in I} \delta_{p_i}$,
    where $(p_i)_i$ is a collection of points in $S$ and $I$ is finite
    or countable. Here, multiple points are allowed. A configuration
    is simple if no multiple points occur, i.e.~$\C(\{x\})\in\{0,1\}$
    for all $x$. In this case we also consider $\C$ as a locally
    finite collection of points, i.e.~the set $\{p_i, i\in
    I\}$.
    Abusing notation we write $\C=\sum_{p\in\C}\delta_p$. For all
    random configurations we consider, multiple points do not occur
    almost surely.

\item For any interval $I \subset \R$, we let $\config(I)$ be the space of point configurations supported on $I$. It is endowed with the initial topology for the family of maps
$\C \mapsto \int \varphi d\C$, for $\varphi$ continuous and compactly supported on $\R$ (topology for which it is a Polish space), and with the associated Borel sigma-algebra.
\item Given a subset $J\subset I$ and $\C\in\config(I)$, we denote by $|_J\C$ the restriction of $\C$ to $J$ (viewed as a measure).
\item If $u \in \R$ and $\C \in \config(\R)$, we let $\C - u$ be the translation of $\C$ by $u$, which is the point configuration corresponding to shifting all the points of $\C$ by $-u$, or equivalently, for any test function $\varphi$
$$
\int \varphi d\left(\C -u\right) := \int \varphi( \cdot + u ) d\C.
$$
\item We will use two ways of enumerating points of a configuration:
\begin{itemize}
\item In a fixed interval $[-R,R]$, we will write $-R \leq \Xx_0 \leq  \Xx_1 \leq \dots \leq \Xx_k \leq \dots \leq R$, enumerating points from the leftmost one.
\item On $\R$, we will write $\dots x_{-k} \leq \dots \leq x_{-1} < 0 \leq x_0 \leq x_1 \leq \dots \leq x_k \leq \dots$, enumerating points starting from the origin.
\end{itemize}
We will need to pass from one enumeration to the other, for example in \eqref{def:SC}.
\item A point process on $I$ is a probability measure on $\config(I)$. For $J\subset I$ and a point process $\P$ on $I$ we denote by $\P_J$ its restriction to $J$, i.e.~$\P_J=(|_J)_*\P$ the point process on $J$ obtained as the image under the restriction map $|_J:\config(I)\to\config(J)$.

\item We let $\Poisson$ be the Poisson point process on $\R$ of intensity $1$.
\end{itemize}

\subsubsection*{Discrepancy}
\begin{definition}[Discrepancy]
If $\C$ is a finite point configuration, we let $|\C|$ be the number of points of $\C$. If $\C$ is a point configuration and $I$ an interval, we let $\C_I$ be the restriction of $\C$ to $I$. In particular, $|\C_I|$ denotes the number of points of $\C$ in $I$. We then define the discrepancy of $\C$ in $I$ as the difference:
\begin{equation}
\label{def:Discr}
\Discr_I(\C) := |\C_I| - |I|.
\end{equation}
\end{definition}

\section{Definition of the free energy}
\label{sec:defoffree}
\subsection{Electric fields and electric energy}
Let us recall that, in the sense of distributions, the following identity holds
$$
- \div \left( \nabla \left(- \log | \cdot | \right) \right) = 2\pi \delta_0 \text{ in } \R \times \R,
$$
which corresponds to the fact that $X \mapsto \frac{1}{2\pi} \log|X|$ is the fundamental solution of Laplace's equation $\Delta f = \delta_0$ in dimension two.

\subsubsection*{Electric fields}
\begin{definition}[Electric fields]
\label{def:electricfields}
Let $I$ be an interval of $\R$. 
\begin{itemize}
\item We call \textit{electric fields on $I$} the set of all vector fields $\El$ in $\Lploc(I \times \R, \R \times \R)$, for some $1 < p < 2$ fixed.
\item Let $\C$ be a finite point configuration in $I$ and $\El$ an electric field on $I$. We say that \textit{$\El$ is compatible with $\C$ in $I$} provided 
\begin{equation}
\label{def:compatible}
-\div(\El) = \c \left(\C - dx \right) \quad \text{in } I \times \R, 
\end{equation}
in the sense of distributions, and using the convention of \eqref{conventionmeasure}. See \eqref{def:Eloc} for an example.
\item If $\El$ is compatible with $\C$ in $I$, for $\eta \in (0,1)$ we define the $\eta$-truncation of the electric field $\El$ as
\begin{equation}
\label{def:Eeta}
\El_{\eta}(X) := \El(X) - \sum_{p \in \C \cap I} \nabla \f_{\eta} (X - (p,0)),
\end{equation}
where $\f_{\eta}$ is the function
$$
\f_{\eta}(x) = \max \left(- \log \left( \frac{|x|}{\eta} \right), 0\right).
$$
\item In particular, if $\C = \sum_{x \in \C} \delta_x$ is a point configuration in $I$ and $E$ is compatible with $\C$, the truncation $\El_{\eta}$ satisfies the equation (compare with \eqref{def:compatible})
\begin{equation}
\label{compa2}
-\div(\El_{\eta}) = \c \left(\sum_{x \in \C} \sigma_{x,\eta} - dx \right) \quad \text{in } I \times \R, 
\end{equation}
where $\sigma_{x,\eta}$ is the “smeared out charge” as in \cite{rougerie2016higher}, the uniform measure on the circle of radius $\eta$ around $x$ (the measure $\sigma_{x,\eta}$ is truly supported on $\R \times \R$).
\end{itemize}
\end{definition}

\begin{definition}[The local electric field]
Let $R > 0$, let $\C$ be a point configuration in $\La_R$.
\begin{itemize}
\item We define the \textit{local electric field generated by $\C$} as
\begin{equation}
\label{def:Eloc}
X \in \R \times \R \mapsto \Eloc(X ; \C ; \La_R) := \int_{\La_R} - \nabla \log |X-(u,0)| (d\C(u) - du). 
\end{equation}
\item For any $\eta > 0$, we define the local energy of $\C$ (truncated at $\eta$) as
\begin{equation}
\label{def:Wel}
\Wel_{\eta}(\C; \La_R) := \frac{1}{2\pi} \int_{\R \times \R} |\Eloc_{\eta}|^2 + |\C| \log \eta.
\end{equation}
\end{itemize}
\end{definition}
\begin{remark}
$\Eloc$ is a model case for the definition of electric fields. Indeed, $\Eloc$ is an electric field compatible with $\C$ in $\La_R$, in the sense of \eqref{def:compatible}, and it is in $\Lploc$ for all $p \in (1,2)$ but fails to be in $L^2$ near each point of $\C$.
\end{remark}

\subsubsection*{The intrinsic energy, monotonicity estimates}
\begin{definition}[The intrinsic energy]
Let $\C$ be a point configuration in $\La_R$. We define the \textit{intrinsic energy of $\C$} as the quantity
\begin{equation}
\label{def:Wint}
\Wint(\C;\La_R) := \iint_{\mdiag} - \log|x-y| (d\C(x) - dx) (d\C(y) - dy)\;,
\end{equation}
where $\diamond$ denotes the diagonal in $\Lambda_R\times\Lambda_R$.
\end{definition}

\begin{lemma}[Monotonicity estimates]
\label{lem:monoto}
Let $\C$ be a finite point configuration in $\La_R$, with exactly $|\La_R|$ points.
\begin{enumerate}
\item The limit $\lim_{\eta \to 0} \Wel_{\eta}(\C;\La_R)$ exists and satisfies
\begin{equation}
\label{WelWint}
\lim_{\eta \to 0} \Wel_{\eta}(\C;\La_R) = \Wint(\C;\La_R).
\end{equation}
\item We have, for $\eta > 0$
\begin{equation}
\label{WelleqWint}
\Wel_{\eta}(\C;\La_R) \leq \Wint(\C;\La_R) + |\La_R| \ErMon(\eta),
\end{equation}
where $\ErMon$ is some function independent of $\C$, satisfying 
\begin{equation}
\label{ErMon0}
\lim_{\eta \to 0} \ErMon(\eta) = 0.
\end{equation}
\item Conversely, we have, for $\eta > 0$
\begin{equation}
\label{WintleqWel}
\Wint(\C;\La_R) \leq \Wel_{\eta}(\C;\La_R) + |\La_R| \ErMon(\eta) + \ErTrun(\C, \eta; \La_R),
\end{equation}
where $\ErMon$ is as above, and $\ErTrun$ satisfies
\begin{equation}
\label{ErTrun}
\ErTrun(\C,\eta;\La_R) \leq \iint_{(x,y) \in \La_R \times \La_R, 0 < |x-y| < 2\eta} - \log \left|x-y \right| d\C(x) d\C(y).
\end{equation}
\end{enumerate}
\end{lemma}
\begin{proof}
The existence of the limit and \eqref{WelWint} is proven e.g. in \cite{sandier20151d}.
The monotonicity estimates \eqref{WelleqWint}, \eqref{WintleqWel} follow from \cite[Lemma 2.3]{petrache2017next}(by sending the parameter $\alpha$ appearing there to $0$).
\end{proof}

\subsubsection*{Electric energy of a point process}
\begin{definition} We introduce the first component of the free energy functional: the electric energy of a point process.
\begin{itemize}
\item Let $\C$ be a point configuration on $\R$. We define the global electric energy of $\C$ as
\begin{equation}
\label{def:tWel}
\tWel(\C) :=  \inf_{E} \left( \lim_{\eta \to 0} \left( \limsup_{R \to \infty} \frac{1}{|\La_R|} \frac{1}{\c} \int_{\La_R \times \R} |\El_{\eta}|^2 + \log \eta \right) \right), 
\end{equation}
where the $\inf$ is taken over electric fields $E$ that are compatible with $\C$ in $\R$, in the sense of Definition \ref{def:electricfields}. As usual, if there is no such electric fields, we let $\tWel(\C) := + \infty$.
\item Let $\P$ be a point process. We define the electric energy of $\P$ as the quantity
\begin{equation}
\label{def:Welec}
\Welec(\P) := \E_{\P} \left[ \tWel(\C) \right].
\end{equation}
\end{itemize}
\end{definition}
We refer to \cite[Section 2.7]{leble2017large} for more details. The following remark is a consequence of \cite[Lemmas 2.3 and 3.8]{leble2017large}.
\begin{remark}
If $\C$ is a point configuration such that $\tWel(\C)$ is finite, then the infimum in \eqref{def:tWel} is attained for exactly one electric field $\El$, and in fact all other compatible fields have infinite energy. If $\P$ is a point process with finite energy, there is a measurable choice $\C \mapsto \El$ of compatible electric fields such that 
$$
\Welec(\P) = \E_{\P} \left[ \lim_{\eta \to 0}  \left( \limsup_{R \to \infty} \frac{1}{|\La_R|} \frac{1}{\c} \int_{\La_R \times \R} |\El_{\eta}|^2 + \log \eta \right) \right].
$$ 
Moreover, if $\P$ is stationary, we can write for any $R > 1$,
\begin{equation}
\label{EnergyStatCase}
\Welec(\P) = \lim_{\eta \to 0} \E_{\P}\left[ \frac{1}{|\La_R|} \frac{1}{\c} \int_{\La_R \times \R} |\El_{\eta}|^2  + \log \eta  \right].
\end{equation}
\end{remark}

\begin{remark}
An alternative definition for the logarithmic energy of a point process $\P$ was introduced in \cite{leble2016logarithmic}, inspired by \cite{MR3046995}. We define the intrinsic energy of $\P$ as 
\begin{equation}
\label{def:Wintglob}
\Wintg(\P) := \liminf_{L \to \infty} \frac{1}{|\La_L|} \E_{\P} \left[ \Wint(\C ; \La_L)\right].
\end{equation}
The energies $\Welec$ and $\Wintg$ are related to each other, in particular it is proven in \cite{leble2016logarithmic}[Prop. 5.1.] that if $\P$ is a stationary process with small discrepancies, for example such that $\sup_L \E_{\P} \left[ \Discr_{\La_L}^2 \right] < \infty$, we have
\begin{equation}
\label{welecwintg}
\Welec(\P) \leq \Wintg(\P),
\end{equation}
and in fact one can obtain some weak form of the converse inequality, via a recovery sequence, we refer to \cite{leble2016logarithmic} for more details.
\end{remark}

\subsection{The specific relative entropy}
\begin{definition}[Specific relative entropy]
Let $\P$ be a stationary point process. The following limit exists in $[0, + \infty]$ and defines the \textit{specific relative entropy} of $\P$ with respect to the Poisson point process.
\label{def:SRE}
\begin{equation}
\label{SRE}
\SRE[\P] := \lim_{R \to \infty} \frac{1}{|\La_R|} \Ent \left[ \P_{| \La_R} \big| \Poisson_{|\La_R} \right].
\end{equation}
\end{definition}
It is a lower semi-continuous, affine function of $\P$. Its unique minimiser -- among stationary processes -- is the Poisson process itself, for which $\SRE[\Poisson] = 0$. We refer e.g. to \cite[Sec. 6.9.2]{friedli2017statistical} for elementary properties of the specific entropy.

\subsection{The free energy}
\begin{definition}[The free energy $\fbeta$]
Let $\beta > 0$ be fixed. For any stationary point process $\P$, we define the free energy $\fbeta(\P)$ as the weighted sum
\begin{equation}
\label{def:fbeta}
\fbeta(\P) := \beta \Welec(\P) + \SRE(\P).
\end{equation}
\end{definition}
It is lower semi-continuous, affine, and has compact sub-level sets. In particular, it has a minimum, and admits a compact set of minimisers (our point is to prove that there is only one).

\section{The screening procedure}
\label{sec:descriscri}
The screening technique has been introduced in \cite{ss2d}, followed by several adaptations in e.g. \cite{rougerie2016higher, petrache2017next}. In this section, we state a version of the procedure suitable for us. In particular, the stationary nature of our problem allows us to bypass the usual step of “finding a good boundary”, and we work with a fixed good boundary. In Claim \ref{claim:discrenearboundary}, we add a technical remark concerning the “new” points produced by the screening construction, that is needed here.

\subsection{Screened fields}
\label{sec:notationscreen}
\begin{definition}[Screened fields]
Let $\C$ be a point configuration in an interval $I$, and let $\El$ be an electric field, compatible with $\C$ in $I$,  i.e. \eqref{def:compatible} holds. We say that $\El$ is screened, and write $\El \in \Screen(\C, I)$ when
\begin{equation} \label{lescreening}
 \El \cdot \vec{\nu} = 0   \quad \text{on } \partial  (I \times \R), 
\end{equation}
where $\vec{\nu}$ is the outer unit normal vector. 
\end{definition}

The original motivation for working with \textit{screened} electric fields is that one can paste such fields together: since their normal component agree on the boundary, they form a globally compatible electric field. In the present paper, we use this screening property rather incidentally, only to compare the “screened” energy with the “local” energy in the proof of Proposition \ref{prop:largeboxapproximation}. Here we mostly use the screening procedure as a way to transform a configuration with possibly more or less than $2R$ points in $[-R,R]$ into a configuration with exactly $2R$ points, in such a way that we keep the “old” configuration on a large segment, and that we do not augment much the energy. 

\subsection{The screening procedure}
\begin{prop}[Screening] 
\label{prop:screening}
Let $s \in \lp 0, \frac{1}{4} \rp$ and let $M > 1$. There exists a universal constant $\etac > 0$ and $R_0=R_0(s,M)>0$ such that for any $\eta\in(0, \etac)$ and any integer $R\geq R_0$, the following holds:\\
\smallskip
  Put $R' := R(1-s)$, $\Old := \La_{R'}$, and
  $\New := \La_R \setminus \Old$.
Let $\C$ be a point configuration in $\La_R$, let $\El$ be compatible with $\C$ in $\Old$, and let $\El_{\eta}$ be the truncation as in \eqref{def:Eeta}.

Assume that 
\begin{align}
\label{definiM} \Mec := \int_{\{-R', R'\} \times [-R, R]} |\El_{\eta}|^2  & \leq M \\
\label{decrvert} \Hec := \frac{1}{s^4 R} \int_{\La_R \times \left( \R \backslash (-\hal s^2 R, \hal s^2 R)\right)} |\El|^2 & \leq 1.
\end{align}
and that furthermore
\begin{equation}
\label{nochargeneargoodboundary}
 \left| \C_{[-R'-2\eta, -R' + 2 \eta]} \right| = 0, \quad \left| \C_{[R'-2\eta, R' + 2 \eta]} \right| = 0.
\end{equation}
Then, there exists a probability measure $\Phiscr_{\C, s, \eta, R}$ on point
configurations in $\La_R$ such that for
$\Phiscr_{\C, s, \eta, R}$-a.e. configuration $\Cscr$, the following
holds:
\begin{itemize}
\item The number of points is given by $|\Cscr_{\La_R}| = |\La_R|$.
\item The configurations $\C$ and $\Cscr$ coincide on $\Old$.
\item There is no point at distance less than $\frac{1}{10}$ of $\{-R, R\}$.
\item The pairs of points of $\Cscr$ which are at distance less than $\eta$ form a subset of the pairs of points of $\C$ with the same property, in particular
\begin{equation}
\label{screeninggoodtrunc}
\iint_{0 < |x-y| < \eta} - \log |x-y| d\Cscr(x) d\Cscr(y) \leq \iint_{0 < |x-y| < \eta} - \log |x-y| d\C(x) d\C(y).
\end{equation}
\item There exists an electric field $\Escr$ in $\Screen(\Cscr, \La_R)$ such that 
\begin{equation} \label{erreurEcrantage} 
\int_{\La_R \times \R} |\Escr_{\eta}|^2 \leq \int_{\La_R \times [-R, R]} |\El_{\eta}|^2  + \ErrScr,
\end{equation}
with an error term $\ErrScr$ bounded by
\begin{equation}
\label{ErrScr}
\ErrScr \preceq |\log \eta| M s R.
\end{equation}
\end{itemize}

Moreover, the restriction of $\Phiscr_{\C, s, \eta, R}$ to $\New$ can be written as the point process given by
\begin{equation}
\label{CscrNew}
\Cscr_{\New} := \sum_{i=1}^{n} \delta_{p_i + r_i \eta},
\end{equation}
where $n :=|\Cscr_{\New}|$ is the number of points of $\Cscr$ in $\New$, the $p_i$'s are points in $\New$, and the $r_i$'s are i.i.d uniform random variables in $\left[-\frac{1}{4},\frac{1}{4}\right]$. The number $n$ and the points $p_i$ are deterministic for $\Phiscr_{\C, s, \eta, R}$, namely they depend only on $\C$.
\end{prop}

\begin{figure}[h]
\label{fig:screenstatement}
\begin{center}
\includegraphics[scale=1]{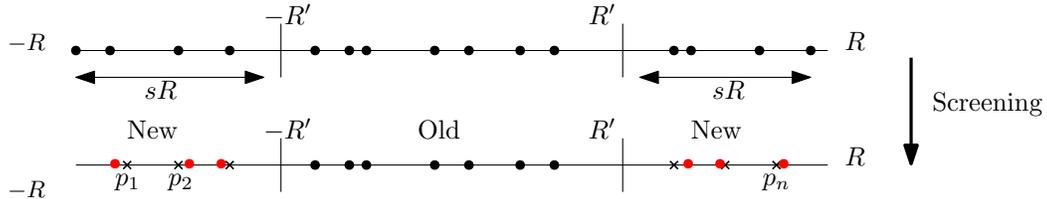}
\end{center}
\caption{A before/after illustration of the screening procedure. The crosses correspond to the deterministic positions $p_i$, which are \textbf{not} points of the configuration, but around which we place “new” points (in red).}
\end{figure}

\begin{remark}
In the screening procedure of \cite{petrache2017next}, designed for being applied to a finite gas, the condition \eqref{definiM} is replaced by a bound of the form
$$
\frac{1}{R} \int_{\La_R \times [-R, R]} |\El_{\eta}|^2 \leq M.
$$
Then a mean value argument is used to find a “good boundary” on which an estimate of the type \eqref{definiM} holds. Taking advantage of the stationary character of the infinite gas, we may skip this step and impose a fixed good boundary. It shortens the procedure a bit, and turns out to be very convenient for the proof of entropy estimates below.

Assuming \eqref{nochargeneargoodboundary} is not mandatory, and not used in \cite{petrache2017next}, but it simplifies the technical details of the proof.
\end{remark}

\subsection{A technical remark on the screening construction}
The set $\New$ consists of two intervals $[-R, -R']$ and $[R',R]$, with $R - R' = sR$. We state here the results for the “left side” $[-R, -R']$, they extend readily to the other side.

\begin{claim}[Position of the new points]
\label{claim:position} 
Let $\kkmax$ be the number of points of $\Cscr$ in $[-R, -R']$ and for $1 \leq k \leq \kkmax$, let $\Xx_k$ be the $k$-th point of $\Cscr$ in $[-R, R']$, starting from the leftmost one. We also write 
$$
\bXx_k = -R + k - 1/2,
$$
which corresponds to the “ideal” position of $\Xx_k$ if the points were regularly spaced. 

\label{claim:discrenearboundary}
\begin{itemize}
\item The number $\kkmax$ satisfies
\begin{equation}
\label{bornekmax}
|\kkmax - sR| \preceq M^{1/2} sR^{1/2}.
\end{equation}
\item For $k$ such that $s^2R \preceq |sR - k|$ (for points that are close to $-R$ and far from the boundary point $-R'$ between $\New$ and $\Old$), we have
\begin{equation}
\label{xkmoinsk} \left|\Xx_k - \bXx_k \right| \preceq \frac{k}{R^{1/2}},
\end{equation}
so in particular we have $\left|\Xx_k - \bXx_k  \right| \preceq \frac{k}{2}$ and $\left|\Xx_k - \bXx_k  \right| \preceq sR^{1/2}$. It also implies
\begin{equation}
\label{discrNewA}
\left|\Discr_{[-R, -R + k]}\right| \preceq \frac{k}{R^{1/2}},
\end{equation}

\item For $k$ such that $|sR - k| \preceq s^2R$ (for points that are close to the boundary point $-R'$ between $\New$ and $\Old$ and far from $-R$), we have:
\begin{equation}
\label{xkmoinskB}
\left| \Xx_k - \bXx_k  \right| \preceq M^{1/2} sR^{1/2}.
\end{equation}
It also implies
\begin{equation}
\label{discrNewB}
\left| \Discr_{[-R, -R + k]} \right| \preceq M^{1/2} sR^{1/2}.
\end{equation}
\end{itemize}
\end{claim}
Proposition \ref{prop:screening} is based on the existing screening procedure. A sketch of the proof with, in particular, a justification of Claim \ref{claim:discrenearboundary}, is given in Section \ref{sec:proofscreening}.

\section{Discrepancy estimate, gaps}
\label{sec:discregaps}
\subsection{A discrepancy estimate}
Let us recall that for an interval $I$, the discrepancy $\Discr_{I}(\C)$ is defined as $|\C_I| - |I|$, hence $\E_{\P} \left[ \Discr_{I}^2 \right]$ can be thought of as the variance, under $\P$, of the number of points in $I$.
\label{sec:discrepancy}
It is observed in \cite[Lemma 3.2]{leble2017large}, that if $\Welec(\P)$ is finite, we have
\begin{equation}
\label{discrbasic}
\limsup_{R \to \infty} \frac{1}{|\La_R|} \E_{\P} \left[\Discr_{\La_R}^2\right] < \infty.
\end{equation}
Moreover, in \cite[Remark 3.3]{leble2017large}, it was proven that if $\P$ is a stationary point process on $\R$ with finite energy, then
\begin{equation}
\label{discrPfiniA}
\liminf_{R \to \infty}  \frac{1}{|\La_R|} \E_{\P} \left[\Discr_{\La_R}^2\right] = 0.
\end{equation}
In fact, an examination of the proof allows for a stronger statement.

\begin{lemma}[The variance of the number of points]
\label{lem:variancesublinear}
Assume that $\Welec(\P)$ is finite. Then
\begin{equation}
\label{discrPfiniB}
\lim_{R \to \infty} \frac{1}{|\La_R|} \E_{\P} \left[\Discr_{\La_R}^2\right] = 0.
\end{equation}
\end{lemma}
We give the proof of Lemma \ref{lem:variancesublinear} in Section \ref{sec:proofvariancesublinear}. The gain from \eqref{discrPfiniA} to \eqref{discrPfiniB} may seem minor, but it turns out to be crucial here (it is also used in a central way in \cite{Leble:2018aa}).

\subsection{Points, gaps and positions}
\subsubsection*{Definition of the gaps}
\begin{definition}[Points and gaps, counted from the origin]
\label{def:gaps}
For any point configuration $\C$ on $\R$, let us enumerate the points of $\C$, counted from the origin, as
\begin{equation}
\label{enumerationB}
\dots < x_{-k} < \dots < x_{-2} < x_{-1} < 0 \leq x_0 < x_1 < \dots < x_k < \dots.
\end{equation}
For $\C$ fixed and $k \in \Z$, we will write $x_k(\C)$ to denote the $k$-th point of $\C$ in the sense of this enumeration. The enumeration \eqref{enumerationB} is well-defined if all the points are simple, which is an event of full measure for all the processes considered here.

We define the sequence of gaps “counted from 0” for $\C$ as the sequence $\{\Gap_k\}_{k \in \Z}$, where
\begin{equation}
\label{GapsFrom0}
\Gap_k = x_{k+1} - x_k.
\end{equation}
If $\C$ is a finite configuration, we let $\Gap_k = + \infty$ after the last point is reached in either direction. 
\end{definition}

\begin{definition}[Position of the first point after translation]
For a given point configuration $\C$, and a real number $u$, we let $\Pos(\C;u)$ be the integer such that
\begin{equation}
\label{def:Pos}
\Pos(\C;u) = m \iff x_{0}(\C - u) = x_{m}(\C) \iff \Gap_0(\C - u) = \Gap_m(\C). 
\end{equation}
In plain words, $\Pos(\C;u)$ is the position in $\C$ of the first point (counted from $0$) of the translated configuration $\C-u$.
\end{definition}
\begin{figure}[ht]
\begin{center}
\includegraphics[scale=1]{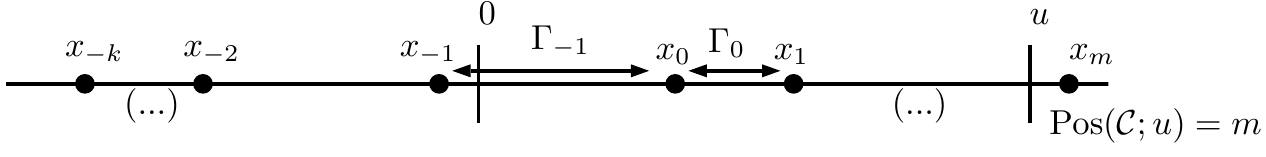}
\end{center}
\end{figure}

\subsubsection*{Average size of the gaps}
\begin{lemma}[Gaps are of order $1$ in average]
\label{lem:gapL2}
Let $0 < \eta < 1/2$, $R > 0$, let $\C$ be a configuration in $\La_R$ that has at least $R/2$ points\footnote{One could e.g. enforce that $R/2$ is an integer, but more generally for $x$ real, having “at least $x$ points” means “having at least $\ceil{x}$ points”, $\sum_{k=0}^{x}$,  means $\sum_{k=0}^{\floor{x}}$ etc.}  in $[0,R]$ and in $[-R, 0]$. We have: 
\begin{equation}
\label{gapsL2}
 \sum_{i=-\frac{R}{2}}^{\frac{R}{2}} \left(\Gap_i\right)^2 \preceq R + \int_{[-R, R] \times \R} |\El_{\eta}|^2.
\end{equation}
\end{lemma}
This result says that gaps are “typically of order 1”, because in view of \eqref{EnergyStatCase} we expect the second term in the right-hand side of \eqref{gapsL2} to be of order $R$. We give the proof of Lemma \ref{lem:gapL2} in Section \ref{sec:proofGapL2}.

\subsubsection*{Deviations estimates for the positions}
Intuitively, since the processes have intensity $1$, the first point of a configuration, counted from the origin, is close to $0$, and the $k$-th point is at a distance $\approx k$. It would imply that the first point of $\C-u$ is $\approx u$, and more generally that the $v$-th point of $\C-u$ is $\approx u + v$. The following lemma shows that these heuristics are correct.
\begin{lemma}
\label{lem:discrUVW}
Let $\P$ be a stationary point process and assume that $\Welec(\P)$ is finite. It holds, for any $u, v$ in $\Z$
\begin{equation}
\label{PposfarA}
\lim_{w \to + \infty} \P \left( x_{\Pos(\C, u) + v}(\C) \geq (u + v) + w\right) = 0.
\end{equation}
Similarly, we have
\begin{equation}
\label{PposfarB}
\lim_{w \to + \infty} \P \left( x_{\Pos(\C, u) + v}(\C) \leq (u + v) - w  \right) = 0.
\end{equation}
\end{lemma}
\begin{proof}
The fact that $ x_{\Pos(\C, u) + v}(\C) \geq (u+v) + w$ implies that there are less than $v$ points between $u$ and $u+v +w$, and hence
$$
\llbr x_{\Pos(\C, u)+v}(\C) \geq w  + u + v \rrbr \subset \llbr \Discr_{[u,u+v+w]} \geq w. \rrbr
$$
In view of \eqref{discrbasic}, and since $\P$ is stationary, the probability of the latter event tends to $0$ as $w \to \infty$, for $v$ fixed, independently of $u$. The proof of \eqref{PposfarB} is identical.
\end{proof}
In particular:
\begin{itemize}
\item Taking $u, v = 0$, so that $\Pos(\C,u) + v = 0$ and sending $w \to \infty$, we get
\begin{equation}
\label{firstoftennotfar}
\lim_{w \to \infty} \P \left( x_0(\C) \geq w \right)  = 0.
\end{equation}
This shows that indeed the first point of a configuration is typically at a bounded distance from $0$.
\item Taking  $v$ fixed, $w = u - v$, and sending $u \to \infty$, we obtain
\begin{equation}
\label{xposugeq2u}
\lim_{u \to \infty} \P \left( x_{\Pos(\C, u)+ v}(\C) \geq 2u \right) = 0.
\end{equation}
This shows that for $v$ fixed and $u$ large, the $v$-th point of $\C-u$ cannot be much further than $2u$. This estimate is far from being sharp, but it is enough for our purposes.
\end{itemize}

\subsection{Distinct stationary processes have distinct gap distributions}
\begin{prop}[Distinct gap distributions I]
\label{prop:gapsdiff}
Let $\Pz, \Pu$ be two stationary point processes such that $\Welec(\Pz)$ and $\Welec(\Pu)$ are finite, and assume that $\Pz \neq \Pu$. Then there exists $c > 0$, an integer $r \geq 1$, and a function $H : \R_+^{2r+1} \to \R$, 
such that
\begin{itemize}
\item $H$ is compactly supported, bounded by $1$, and $1$-Lipschitz with respect to the $\|\cdot\|_{1}$ norm on $\R_+^{2r+1}$, i.e.
\begin{equation}
\label{HLip}
\| H \|_{\infty} \leq 1, \quad \left| H(a_{-r}, \dots, a_{r}) - H(b_{-r}, \dots, b_{r})\right| \leq \sum_{i=-r}^{r} |a_{i} - b_i|.
\end{equation}
\item $H$ detects the difference of distribution of the gaps of $\Pz$ and $\Pu$, namely
\begin{equation}
\label{gapsdiff}
\E_{\Pz}\left[ H(\Gap_{-r}, \dots, \Gap_r) \right] - \E_{\Pu}\left[ H(\Gap_{-r}, \dots, \Gap_r) \right] > c.
\end{equation}
\end{itemize}
\end{prop}
We postpone the proof of Lemma \ref{prop:gapsdiff} to Section \ref{sec:proofgapsdiff}.

\begin{prop}[Distinct gap distributions II] 
\label{prop:gaps2}
Let $\Pz, \Pu$ be two stationary point processes such that
$\Welec(\Pz)$ and $\Welec(\Pu)$ are finite, and assume that
$\Pz \neq \Pu$. There exists $\g > 0$ such that for $R$ large enough,
any coupling $Q$ of the restrictions of $\Pz,\Pu$ to $\La_R$, and any
random variable $S$ bounded by $R^{1/2}$, then
\begin{equation}
\label{EPigaps}
\E_{\Qo} \left[ \sum_{i=-R/2}^{R/2} \frac{\left|\Gap_i(\Cz) - \Gap_{i+S}(\Cu) \right|^2}{|\Gap_i(\Cz)|^2 + |\Gap_{i+S}(\Cu)|^2} \right] \geq \g R.
\end{equation}

\end{prop}
We give the proof of Proposition \ref{prop:gaps2} in Section \ref{sec:proofgaps2}. Intuitively, the reason that \eqref{EPigaps} holds true is the following: we know that the gaps are typically of order $1$ (see Lemma \ref{lem:gapL2}), so the denominator is of order $1$, and $\Pz, \Pu$ are stationary so the difference between their gaps distribution on a given interval should be proportional to its length.

Since the quantity $\frac{\left|\Gap_i(\Cz) - \Gap_{i+S}(\Cu) \right|^2}{|\Gap_i(\Cz)|^2 + |\Gap_{i+S}(\Cu)|^2}$ is bounded, the following extension of Proposition \ref{prop:gaps2} is straightforward.
\begin{coro}
\label{coro:gaps}
In the situation of Proposition \ref{prop:gaps2}, if $A, B$ are events such that
$$
\Pz(A) \geq 1 - \frac{\g}{100}, \quad \Pu(B) \geq 1 - \frac{\g}{100},
$$
and if $\tQo$ is a coupling of the restrictions to $\La_R$ of $\Pz$ conditioned to $A$ and $\Pu$ conditioned to $B$, then 
\begin{equation*}
\E_{\tQo} \left[ \sum_{i=-R/2}^{R/2} \frac{\left|\Gap_i(\Cz) - \Gap_{i+S}(\Cu) \right|^2}{|\Gap_i(\Cz)|^2 + |\Gap_{i+S}(\Cu)|^2} \right] \geq \hal \g R.
\end{equation*}
\end{coro}
\begin{proof}
Indeed, we can lift the coupling $\tQo$ of the conditioned processes (restricted to $\La_R$) to a coupling of $\Pz, \Pu$ (restricted to $\La_R$) by defining
$$
\Qo := \Pz(A)\Pu(B)\cdot \tQo +\1_{(A\times B)^c}\cdot\Pz\big|_{\La_R}\otimes\Pu\big|_{\La_R},
$$
and we apply Proposition \ref{prop:gaps2} to $\Qo$. Letting 
$$
G : = \sum_{i=-R/2}^{R/2} \frac{\left|\Gap_i(\Cz) - \Gap_{i+S}(\Cu) \right|^2}{|\Gap_i(\Cz)|^2 + |\Gap_{i+S}(\Cu)|^2},
$$
which is a quantity bounded by $R$, we have $\E_{\Qo} [ G ] \geq \g R$, but also
$$
\E_{\tQo} [ G ]= \E_{ \Qo}[\1_{A\times B} G ]\frac{1}{\Pz(A)\Pu(B)}\geq \E_{\Qo}[\1_{A\times B} G] \geq \E_{\Qo}[G] - \Qo( (A \times B)^c)R,
$$
and using the assumptions on $\Pz(A), \Pu(B)$ we get the result.
\end{proof}

\section{Large box approximation}
In this section, we introduce an approximation scheme: given a point process $\P$, and parameters $\epsilon, s > 0$, for $R$ large enough we are able, after conditioning on a likely event $\EveR$, to restrict $\P$ in $\La_R$, to modify the configurations near the edges of $\La_R$, and to produce a “local candidate” $\tPR$ in $\La_R$ which
\begin{enumerate}
\item “Looks like” $\P$ (in a sense that is controlled by the parameter $s$).
\item Has an energy, entropy comparable to those of $\P$, up to small errors (controlled by $\epsilon$).
\item Has various good discrepancies bounds (controlled by $s$).
\item Has almost surely $2R$ points in $\La_R$.
\end{enumerate}
Points 1 and 4 are obtained through the screening procedure described in Section \ref{sec:descriscri}. Points 2 and 3 are guaranteed by conditioning on a good event, and using properties that hold for processes with finite energy, as mentioned in Section \ref{sec:discregaps}. The first step is to check that certain conditions, related to the applicability of the screening procedure, or to the good controls on the discrepancy that we want to enforce, are often satisfied under $\P$, if $\P$ has finite energy.

\label{sec:largeboxscreening}
\subsection{Conditions that are often satisfied}
Let $\P$ be a stationary point process such that $\Welec(\P)$ is finite.  We recall that to $\P$-almost every configuration is associated an electric field $\El$ such that \eqref{EnergyStatCase} holds.
\subsubsection*{Good energy}
This control is related to the condition \eqref{definiM} in the statement of the screening procedure.
\begin{claim}[Good energy]
\label{claim:goodenergy}
For all $\eta$ in $(0,1)$, for all $\delta > 0$, for all $M$ large enough (depending on $\P, \eta, \delta$), for all $s \in \lp 0, \frac{1}{4} \rp$ for all $R > 1$, letting $R' = R(1-s)$ (as in Proposition \ref{prop:screening}), the event
\begin{equation*}
\EventEnergy(R, s, M, \eta) := \llbr \int_{\{-R', R'\} \times [-R, R]} |\El_{\eta}|^2  \leq M \rrbr 
\end{equation*}
satisfies
\begin{equation*}
\P\left( \EventEnergy(R, s, M, \eta) \right) \geq 1 - \delta.
\end{equation*}
\end{claim}
\begin{proof}
Since $\P$ is stationary and has finite energy we have, in view of \eqref{EnergyStatCase}, for any fixed $\eta$, 
for any $x > 0$
$$
\E_{\P} \left[  \int_{\{-x, x\} \times \R} |\El_{\eta}|^2 \right] 
= \E_{\P} \left[ \frac{1}{R} \int_{\La_R \times \R} |\El_{\eta}|^2 \right] < + \infty,
$$
and the claim follows by applying Markov's inequality.
\end{proof}

\subsubsection*{Good decay on the vertical axis}
This control is related to the condition \eqref{decrvert} in the statement of the screening procedure.
\begin{claim}[Good decay of the field the vertical axis]
\label{claim:gooddecay}
For all $\delta > 0$, for all $s > 0$, for all $R$ large enough (depending on $\P, \delta, s$), the event
\begin{equation*}
\EventDecay(R,s) := \llbr \frac{1}{s^4 R} \int_{[-R, R] \times (\R \backslash (-s^2 R, s^2 R))} |\El|^2 \leq 1 \rrbr
\end{equation*}
satisfies
\begin{equation*}
\P\left( \EventDecay(R,s) \right) \geq 1 - \delta.
\end{equation*}
\end{claim}
\begin{proof}
For any $R$, for any $t > 0$, since $\P$ is stationary we have
$$
\E_{\P} \left[ \frac{1}{R} \int_{[-R, R] \times (\R \backslash (-t, t))} |\El|^2  \right] = \E_{\P} \left[ \int_{[-1, 1] \times (\R \backslash (-t, t))} |\El|^2  \right] < + \infty.
$$
On the other hand, by dominated convergence we have
$$
\lim_{t \to + \infty} \E_{\P} \left[ \int_{[-1, 1] \times (\R \backslash (-t, t))} |\El|^2  \right] = 0.
$$
Therefore, for any given $\delta, s$, we see that for $t \geq t_0$ large enough (depending on $\P, \delta, s$), it holds
$$
\E_{\P} \left[ \int_{[-1, 1] \times (\R \backslash (-t, t))} |\El|^2 \right] \leq \delta s^4\;.
$$
Then for any $R \geq R_0 := \frac{t_0}{s^2}$ (depending on $\P, \delta, s$), we have
\begin{equation*}
\P \left(\frac{1}{s^4 R} \int_{[-R, R] \times (\R \backslash (-s^2 R, s^2 R))} |\El|^2 > 1  \right) \leq \frac{1}{s^4} \E_{\P} \left[ \int_{[-1, 1] \times (\R \backslash (-s^2R, s^2R))} |\El|^2 \right] \leq \delta,
\end{equation*}
which proves the claim.
\end{proof}

\subsubsection*{The “good field” event}
\begin{definition}
We introduce the event $\GoodField(R, s, M, \eta)$ as the set of all configurations $\C$ such that there exists a field $\El$ satisfying the following three conditions:
\begin{enumerate}
\item $\El$ is compatible with $\C$ in $\Old := \La_{R(1-s)} = \La_{R'}$.
\item $\El$ satisfies the “good energy” estimate
$$
 \int_{\{-R', R'\} \times [-R, R]} |\El_{\eta}|^2  \leq M
$$
\item $\El$ satisfies the “good decay” estimate
$$
\frac{1}{s^4 R} \int_{[-R, R] \times (\R \backslash (-s^2 R, s^2 R))} |\El|^2 \leq 1
$$
\end{enumerate}
\end{definition}
\begin{claim}
\label{claim:goodfield}
By construction, the event $\GoodField(R, s, M, \eta)$ depends only on the restriction of $\C$ to $\La_{R'}$. Since the global field (compatible with $\C$ on the whole real line) is always a candidate, and in view of the definitions above, we have
\begin{equation*}
\P \left(\GoodField(R, s, M, \eta) \right) \geq \P \left( \EventEnergy(R, s, M, \eta) \cap   \EventDecay(R,s) \right),
\end{equation*}
so in particular, combining Claim \ref{claim:goodenergy} and Claim \ref{claim:gooddecay}.
For all $\delta > 0$, for all $M$ large enough, for all $s > 0$, for all $R$ large enough, we have 
\begin{equation*}
\P \left(\GoodField(R, s, M, \eta) \right)  \geq 1 - \delta.
\end{equation*}
\end{claim}

\subsubsection*{Good discrepancy estimates}
The condition \eqref{NoPoint} is related to the assumption \eqref{nochargeneargoodboundary} in the screening procedure, and the two other conditions are useful estimates that we want to enforce.
\begin{claim}[Good discrepancy estimates]
\label{claim:gooddiscr}
Let us introduce the following events:
\begin{align}
\label{NoPoint} \NoPoint(R,s, \eta) & := \llbr \left| \C_{[-R' - 2\eta, -R' + 2\eta]} \right| = \left| \C_{[R' - 2\eta, R' + 2\eta]} \right|  = 0 \rrbr \\
\label{LeftRight} \LeftRight(R, s) & := \llbr \left|\Discr_{[-R', 0]} \right| + \left| \Discr_{[0, R']} \right| \leq s R^{1/2} \rrbr, \\
\label{TotDiscr} \TotDiscr(R, s) & := \llbr \sum_{i=-R'}^{R'} \left| \Discr_{[-R', i]} \right| \leq s^2 R^{3/2}, \rrbr.
\end{align}
and we let
$\displaystyle{\EventDiscr(R, s, \eta)  := \NoPoint(R, s, \eta) \cap \LeftRight(R,s) \cap \TotDiscr(R, s).}$

Then for all $\delta > 0$, for all $\eta > 0$ small enough (depending on $\delta$), for all $s \in \lp 0, \frac{1}{4} \rp$, for all $R$ large enough (depending on $\P, \delta, s$), we have
\begin{equation*}
\P \left( \EventDiscr(R, s ,\eta) \right) \geq 1 - \delta.
\end{equation*}
\end{claim}
\begin{proof}
For $\NoPoint$, we recall that $\P$ has intensity $1$, so for all $s, R$ and $\eta > 0$.
$$
\E_{\P} \left[ \left| \C_{[-R' - 2\eta, -R' + 2\eta]} \right| \right] = 4\eta.
$$
Of course the number of points in a given interval is an integer valued random variable. We thus have
$$
\P \left[ \left| \C_{[-R' - 2\eta, -R' + 2\eta]} \right| > 0 \right] = \P \left[ \left| \C_{[-R' - 2\eta, -R' + 2\eta]} \right| \geq 1 \right] \leq 4 \eta,
$$
and we then take $\eta$ small enough (depending only on $\delta$).

For $\LeftRight$, we use Lemma \ref{lem:variancesublinear} and Markov's inequality.
For $\TotDiscr$, we note that by  Lemma \ref{lem:variancesublinear} we have 
  \begin{align*}
    \frac{1}{R^{3/2}}\E_{\P} \left[ \sum_{i = -R'}^{R'} |\Discr_{[-R',i]}| \right] \leq
      \frac{1}{R}\sum_{i = -R'}^{R'}\sqrt{\frac{1}{R}\E_{\P} \left[ |\Discr_{[-R',i]}|^2 \right]}
    \longrightarrow 0\;, \quad\text{as }R\to\infty\;,
  \end{align*}
and use again Markov's inequality.
\end{proof}

\subsubsection*{Good truncation error}
\begin{claim}[Good truncation error]
\label{claim:goodtruncationerror}
For all $\delta, \epsilon > 0$, for all $\eta \in (0,1)$ small enough (depending on $P, \delta, \epsilon$), and for all $R >  1$, the event
\begin{equation}
\label{def:EventTrunc}
\EventTrunc(R, \eta, \epsilon) := \llbr \iint_{\{x,y \in \La_R \times \La_R, 0 < |x-y| < 2\eta\}} - \log \left|x-y \right| d\C(x) d\C(y) \leq \frac{\epsilon}{100} R \rrbr
\end{equation}
satisfies
\begin{equation*}
\P \left(\EventTrunc(R, \eta, \epsilon)\right) \geq 1 - \delta.
\end{equation*}
\end{claim}
\begin{proof}
By stationarity we have, for $R > 1$, 
\begin{multline*}
\E_{\P} \left[ \iint_{\{(x,y) \in \La_R \times \La_R, 0 < |x-y| < 2\eta\}} - \log \left|x-y \right| d\C(x) d\C(y) \right] \\
\preceq R \times \E_{\P} \left[ \iint_{\{(x,y) \in \La_1 \times \La_1, 0 < |x-y| < 2\eta\}} - \log \left|x-y \right| d\C(x) d\C(y) \right].
\end{multline*}
From \cite[Lemma 3.5.]{leble2017large}, we know that 
$$
\lim_{\eta \to 0} \E_{\P} \left[ \iint_{\{(x,y) \in \La_1 \times \La_1, 0 < |x-y| < 2\eta\}} - \log \left|x-y \right| d\C(x) d\C(y) \right] = 0,
$$
and we use Markov's inequality to conclude.
\end{proof}

\subsection{Large box approximation}
We now state the approximation result.

\begin{prop}[Large box approximation]
\label{prop:largeboxapproximation}
Let $\P$ be a stationary process such that $\fbeta(\P)$ is finite. Let $\epsilon > 0$ be fixed. There exists an “energy threshold” $M$, depending on $\P$ and $\epsilon$, such that for any $s > 0$ small enough, for all integer $R$ large enough (depending on $\P, \epsilon, s$), there exists an event $\EveR$ depending only on the restriction of configurations to $\La_R$  such that 
\begin{equation}
\label{PEverbig}
\P(\EveR) \geq 1 - \epsilon,
\end{equation}
and there exists a probability measure $\tPR$ on $\config(\La_R)$ such that:
\begin{itemize}
\item The configurations have exactly $2R$ points in $\La_R$, $\tPR$-almost surely.
\item The restrictions to $\Old = \La_{R(1-s)}$ of the process $\tPR$ and of the process $\P$ conditioned on $\EveR$ coincide. More precisely, if $|_\Old$ denotes the restriction of a configuration to $\Old$, then we have
$${|_\Old}_*\tPR = {|_\Old}_*\P[\cdot|\EveR]\;.$$ 
\end{itemize}
Moreover, the following inequalities hold:
\begin{align}
\label{goodEner}
  \frac{1}{|\La_R|} \E_{\tPR} \left[ \Wint(\C; \La_R) \right]  & \leq \Welec(\P) + \epsilon, \\
  \label{goodEnt} \frac{1}{|\La_R|} \Ent\left[\tPR | \Poisson_{|\La_R} \right] & \leq \SRE(\P) + \epsilon.
\end{align}

For $\tPR$-a.e. configuration $\Cscr$:
\begin{itemize}
\item The points in $\New$ follow the conclusions of Claim \ref{claim:position}. 
\item Letting $S$ be the random variable such that
\begin{equation}
\label{def:SC}
x_0(\Cscr) = \Xx_{R+S(\Cscr)}(\Cscr),
\end{equation}
where $x_0(\Cscr)$ denotes the first point of $\Cscr$ to the right of $0$ as in \eqref{enumerationB} and $\Xx_{k}(\Cscr)$ denotes the $k$-th point of $\Cscr$ to the right of $-R$, as in Claim \ref{claim:position},
we have 
\begin{equation}
\label{controlS}
|S| \leq s M^{1/2} R^{1/2}, \quad \tPR\text{-a.s.}
\end{equation}
\item We have 
\begin{equation}
\label{discrpostscr}
\sum_{i=-R}^{R} \left| \Discr_{[-R,i]}(\Cscr) \right | \preceq s^2 R^{3/2}.
\end{equation}
\end{itemize}
\end{prop}
Let us comment on \eqref{def:SC}, \eqref{controlS}. For a given configuration with $2R$ points in $\La_R$, we consider two different ways of enumerating points: with $\Xx_0$ starting from the left, as in Claim \ref{claim:position}, and with $x_0$ being the first positive point, as in \eqref{enumerationB}. There should be roughly $R$ points in $[-R, 0]$, so $x_0$ should correspond to $\Xx_R$. The “shift” $S$ defined in \eqref{def:SC} quantifies the deviation to this guess, and \eqref{controlS} affirms that this deviation is small.

The proof relies on the screening procedure of Proposition \ref{prop:screening}. First, we need to show that we can apply this procedure with high probability.

\subsubsection*{Finding a good event.} 
Let $\epsilon > 0$ be given.

\textbf{The truncation.} First, since $\P$ is stationary, we know from \eqref{EnergyStatCase} that, for any $R > 1$,
$$
 \lim_{\eta \to 0} \E_{\P}\left[ \frac{1}{|\La_R|} \frac{1}{\c} \int_{\La_R \times \R} |\El_{\eta}|^2 + \log \eta  \right] = \Welec(\P),
$$
and the limit is uniform in $R$ because the quantity (the expectation) whose limit is taken is independent of $R$, by stationarity. In particular for $\eta$ small enough (depending only on $\P$), for all $R > 1$, we have
\begin{align}
\label{etatruncA}
& \left| \E_{\P}\left[ \frac{1}{|\La_R|} \frac{1}{\c} \int_{\La_R \times \R} |\El_{\eta}|^2  + \log \eta  \right] - \Welec(\P)\right| \leq \frac{\epsilon}{100},
\end{align}
Another important feature for choosing a suitable truncation is that the error terms in the monotonicity estimates of Lemma \ref{lem:monoto} should be small. By Claim \ref{claim:goodtruncationerror}, we may take some $\eta > 0$ (depending on $\P, \epsilon$) such that
\begin{equation*}
\P \lp \EventTrunc\lp R, \eta, \epsilon \rp \rp \geq 1 - \frac{\epsilon}{100}. 
\end{equation*}
and by taking $\eta$ smaller if needed, we may also impose that
\begin{equation}
\label{choiceeta}
\ErMon(\eta) < \frac{\epsilon}{100}, \quad \eta < \etac,
\end{equation}
where $\ErMon$ is the error term in \eqref{WelleqWint}, \eqref{WintleqWel}, and $\etac$ is the constant in Proposition \ref{prop:screening}.

\textbf{The energy threshold.} We now fix an energy threshold $M$. We take it high enough such that most configurations have an energy at most $M$. By Claim \ref{claim:goodfield}, for $M$ large enough depending on $\P, \epsilon$, for any $s \in \left(0, \frac{1}{4} \right)$, for $R$ large enough (depending on $\P, s, M, \epsilon$) we have
\begin{equation}
\label{goodfield}
\P\left( \GoodField(R, s, M, \eta) \right) \geq 1 -  \frac{\epsilon}{100}.
\end{equation}

\textbf{The size of the screening zone.} Next, we fix the size of the screening zone. The screening procedure affects a small region near the endpoints, and has an energy cost proportional to this size, controlled by $s$. We choose $s \in \lp 0, \frac{1}{4}\rp$ small enough such that for all $R$
\begin{equation}
\label{choiceofs}
 \ErrScr \leq  \frac{\epsilon R}{100},
\end{equation}
where $\ErrScr$ is the error term in \eqref{erreurEcrantage}. This choice of $s$ depends only on $\eta, M$.

\textbf{Various controls.}
Claim \ref{claim:gooddiscr} ensures that, for $R$ large enough, depending on $\P, \epsilon, s$, we have
\begin{equation*}
\P \left( \EventDiscr(R, s, \eta) \right) \geq 1 - \frac{\epsilon}{100}.
\end{equation*}

\textbf{The good event.} We can now define the event.
\begin{definition}[The event $\EveR$]
Let $\epsilon, \eta, M, s$ be as chosen in the previous paragraphs,
and for all $R > 0$ we let $\EveR$ be the intersection of the events
$$
\EventTrunc(R, \eta, \epsilon), \GoodField(R, s, M, \eta), \EventDiscr(R, s, \eta), 
$$ 
\end{definition}
By a union bound, for $R$ large enough depending only on
$\P, \epsilon, s$ we obtain $\P\left(\EveR\right) \geq 1 - \epsilon$,
which ensures that \eqref{PEverbig} holds. By Claim
\ref{claim:goodfield}, $\GoodField$ is measurable with respect to the
restriction to $\La_R$. Moreover $\EventTrunc, \EventDiscr$ clearly
depend only on $\C$ in $\La_R$. Therefore, $\EveR$ depends only on the
restriction to $\La_R$.

\subsubsection*{Defining the modified process} Let $\C$ be in $\EveR$. By definition of $\GoodField(R, s, M, \eta)$, there exists an electric field $\El$ compatible with $\C$ on $\Old \subset \La_R$ and satisfying
\begin{align*}
\Mec := \int_{\{-R', R'\} \times [-R, R]} |\El_{\eta}|^2  & \leq M,  \\
\Hec := \frac{1}{s^4 R} \int_{[-R, R] \times \R \backslash (-s^2 R, s^2 R)} |\El|^2 & \leq 1.
\end{align*}
In the sequel, we choose a field $\El$ that satisfies these conditions \textit{and has minimal energy} i.e. such that $\Mec$ is minimal. In particular, this electric field $\El$ depends only on $\C$ in $\Old$.

The assumption \eqref{nochargeneargoodboundary} is also satisfied by definition of $\EventDiscr(R, s, \eta)$, which is included in $\NoPoint(R, s, \eta)$ as defined in \eqref{NoPoint}.

Hence for all $R$ large enough (depending on $\P, \epsilon, s$) we may apply the screening procedure to $\C$. We let $\Phiscr_{s, \eta, R}(\C)$ be the resulting probability measure on $\config(\La_R)$.

\begin{definition}[Definition of $\tPR$]
We define $\tPR$ as the mixture of the $\Phiscr_{s, \eta, R}(\C)$ for $\C$ in $\EveR$, weighted by $\P(\C)$, i.e.
\begin{equation}
\label{def:tPR}
\tPR := \frac{1}{\P(\EveR)} \int \Phiscr_{s, \eta, R}(\C) \1_{\C \in \EveR} d\P(\C).
\end{equation}
Equivalently, if $F$ is a bounded function on $\config(\La_R)$, we let
\begin{equation}
\label{def:tPRbis}
\E_{\tPR} [F] = \frac{1}{\P(\EveR)} \int \E_{\Phiscr_{s, \eta, R}(\C)} [F] \1_{\C \in \EveR} d\P(\C).
\end{equation}
\end{definition}

We may already check that:
\begin{itemize}
\item By construction, the screened configurations all have $2R$ points in $\La_R$.  
\item Still by construction, for a given $\C \in \EveR$, all the configurations $\Cscr$ in $\Phiscr_{s, \eta, R}(\C)$ coincide with $\C$ inside $\Old = \La_{R'}$. Thus, from the definition \eqref{def:tPRbis}, we see that the restrictions to $\La_{R'}$ of the process $\tPR$  and of the process $\P$ conditioned to  $\EveR$ coincide, in other words
$${|_{\La_{R'}}}_*\tPR = {|_{\La_{R'}}}_*\P[\cdot|\EveR]\;.$$
\end{itemize}

\subsubsection*{Energy estimate} We prove \eqref{goodEner}. Let $\C$ be in $\EveR$. By the screening procedure, for all $\Cscr$ in the support of $\Phiscr_{s, \eta, R}(\C)$, there exists a screened electric field whose energy, in view of the controls \eqref{erreurEcrantage}, \eqref{ErrScr} on the energy after screening, and by the choice of $s$ as in \eqref{choiceofs}, is bounded as follows:
\begin{equation} \label{errorscreeningwellchosen} 
\int_{\La_R \times \R} |\Escr_{\eta}|^2 \leq  \int_{\La_R \times [-R, R]} |\El_{\eta}|^2  + \frac{\epsilon R}{100}.
\end{equation}
Let $\Eloc$ be the local electric field generated by $\Cscr$ in $\La_R$, as in definition \eqref{def:Eloc}. By “minimality of the local energy”\footnote{Remark: this is the only moment where we really use the fact that the new electric field that we have produced is \textit{screened}.}, see e.g. \cite[Lemma 3.10]{leble2017large}, we have
\begin{equation}
\label{minimalite}
\int_{\R \times \R} |\Eloc_{\eta}|^2 \leq \int_{\La_R \times \R} |\Escr_{\eta}|^2.
\end{equation}
Combining \eqref{errorscreeningwellchosen} and \eqref{minimalite}, we get
\begin{equation}
\label{ElocvsEeta}
\int_{\R \times \R} |\Eloc_{\eta}|^2 \leq \int_{\La_R \times [-R, R]} |\El_{\eta}|^2  + \frac{\epsilon R}{100}.
\end{equation}
Next, the monotonicity inequality \eqref{WintleqWel} reads
$$
\Wint(\Cscr; \La_R) \leq \frac{1}{\c} \int_{\R \times \R} |\Eloc_{\eta}|^2 + |\Cscr| \log \eta  + |\Cscr| \ErMon(\eta) + \ErTrun(\Cscr, \eta; \La_R).
$$
By construction we have $\ErTrun(\C, \eta; \La_R) \leq \frac{\epsilon R}{100}$, and we know by \eqref{screeninggoodtrunc} that the screening procedure does not augment the truncation error. Moreover, $\eta$ has been chosen such that $\ErMon(\eta) \leq \frac{\epsilon}{100}$. Using the fact that $|\Cscr| = |\La_R| = 2R$, we obtain
\begin{equation}
\label{WintscrElocscr}
\Wint(\Cscr; \La_R) \leq \frac{1}{\c} \int_{\R \times \R} |\Eloc_{\eta}|^2  + 2R \log \eta  + \frac{4 R \epsilon}{100}.
\end{equation}
Combining \eqref{ElocvsEeta} and \eqref{WintscrElocscr} yields
\begin{equation}
\label{WintscrEeta}
\Wint(\Cscr; \La_R) \leq \frac{1}{\c} \int_{\La_R \times [-R, R]} |\El_{\eta}|^2 + 2R \log \eta + \frac{5 R \epsilon}{100}.
\end{equation}

Moreover, we obtain, in view of the control \eqref{WintscrEeta} on $\Wint$ and the definition \eqref{def:tPRbis} of $\tPR$,
\begin{multline}
\label{calculenergy}
\E_{\tPR} \left[ \Wint(\Cscr; \La_R) \right] = \frac{1}{\P(\EveR)} \int \E_{\Phiscr_{s, \eta, R}(\C)} \left[\Wint(\Cscr;\Lambda_R)\right] \1_{\EveR}(\C) d\P(\C) \\
\leq \frac{1}{\P(\EveR)} \int \left( \frac{1}{\c} \int_{\La_R \times [-R, R]} |\El_{\eta}|^2 + 2R \log \eta \right)  \1_{\EveR}(\C) d\P(\C) + \frac{5 R \epsilon}{100}.
\end{multline}
Then since $\P(\EveR) \geq 1 - \epsilon$, and \eqref{etatruncA} holds, we write
$$
\frac{1}{\P(\EveR)} \int \left( \frac{1}{\c} \int_{\La_R \times [-R, R]} |\El_{\eta}|^2 + 2R \log \eta \right)  \1_{\EveR}(\C) d\P(\C)  \leq |\La_R| (1 + \epsilon) \left( \Welec(\P) + \epsilon \right),
$$
so we obtain
$$
\frac{1}{|\La_R|} \E_{\tPR} \left[ \Wint(\Cscr; \La_R) \right] \leq  \Welec(\P) + \epsilon \Welec(\P) + O(\epsilon).
$$
Up to working with a smaller $\epsilon'$ instead of $\epsilon$, we may ensure that $\epsilon' \Welec(\P) + O(\epsilon') \leq \epsilon$ and so \eqref{goodEner} holds.

\subsubsection*{Entropy estimate} We prove \eqref{goodEnt}. Let us first recall a general disintegration principle. 
\begin{lemma}[Relative entropy and disintegration]
Let $X,Y$ be Polish spaces, let $\mu,\pi$ be two
probability measures on $X$ and let $T:X\to Y$ be a measurable map. Let $\bar\mu=T_{*}\mu$ and $\bar\pi=T_{*}\pi$ and let $\mu(\cdot|T=y)$ and $\pi(\cdot|T=y)$ denote the regular conditional probabilities, i.e.~we have the disintegration
\begin{align*}
  \mu(A)= \int_Y\mu(A|T=y)d\bar\mu(y)\;,\qquad\pi(A)= \int_Y\pi(A|T=y)d\bar\pi(y)\quad \forall A\;. 
\end{align*}
Then we have that
\begin{align}\label{eq:ent-desint}
  \Ent[\mu|\pi]= \Ent[\bar\mu|\bar\pi] +\int_Y \Ent\big[\mu(\cdot|T=y)|\pi(\cdot|T=y)\big]d\bar\mu(y).
\end{align}
\end{lemma}
This can be verified by a direct computation.
Since the relative entropy w.r.t.~a probability is non-negative, we have in particular
\begin{align}\label{eq:ent-push}
  \Ent[\mu|\pi]\geq \Ent[\bar\mu|\bar\pi]\;.
\end{align}

\noindent
Let us briefly summarize some important properties from the construction of $\tPR$:
\begin{itemize}
\item First, we condition $\P$ to the “good event” $\EveR$, with $\P \lp \EveR \rp \geq 1 - \epsilon$, then we restrict to $\La_R$. More precisely, we consider the process
\begin{equation}
\label{def:chP} 
\chP_{R}={|_{\Lambda_R}}_*\P[\cdot|\EveR],
\end{equation}
namely the restriction to $\La_R$ of the process $\P$ conditioned to $\EveR$.
\item Then we apply the screening procedure. Given a configuration $\C$ in $\EveR$, we find (see Figure \ref{fig:screenstatement}):
 \begin{itemize}
\item A number $n = n(\C)$, such that $n(\C) + |\C|_\Old| = 2R$.
\item Points $p_1(\C), \dots, p_n(\C)$ in $\New$ such that the segments $[p_i - \eta, p_i + \eta]$ are disjoint and do not intersect $[-R', R']$.
\end{itemize}
\item We define $\Phiscr(\C)$ by adding $n(\C)$ points placed independently, one into each interval of the form $[p_i - \eta, p_i + \eta]$.
\item We define $\tPR$ as the weighted average of $\Phiscr(\C)$ for $\C$ in $\EveR$, as in \eqref{def:tPR}. 
\end{itemize}

\paragraph{\textbf{Step 1:}} We first estimate the entropy of $\chP_{R}$ and claim that, for some constant $C$ (independent of $\P, R$)
  \begin{equation}
  \label{eq:ent-restr}
    \Ent[\chP_{R} |\Poisson_{\La_R}] \leq (1+2\epsilon) \Ent[\P_{\La_R}|\Poisson_{\La_R}] + C.
  \end{equation}
  \noindent Indeed, if $\rho$ is the density of
  $\P_{\La_R}:={|_{\La_R}}_*\P$ with respect to $\Poisson_{\La_R}$, the
  density of $\chP_{R}$ is given by $\1_E \rho/\P(E)$, where we set
  $E=\EveR$ for brevity. Thus we have
\begin{equation*}
  \Ent[\chP_{R}|\Poisson_{\La_R}] = \int_E\frac{\rho}{\P(E)}\log\lp\frac{\rho}{\P(E)}\rp d\Poisson_{\La_R} =
- \log \P(E) + \frac{1}{\P(E)}\left( \Ent\left[\P_{\La_R}|\Poisson_{\La_R}\right] -\int_{E^c}\rho\log\rho d\Poisson_{\La_R} \right).
\end{equation*}
By Jensen's inequality, we have
\begin{align*}
  \frac1{\Poisson_{\La_R}(E^c)}\int_{E^c}\rho\log\rho d\Poisson_{\La_R} \geq \frac{\P_{\La_R}(E^c)}{\Poisson_{\La_R}(E^c)}\log\frac{\P_{\La_R}(E^c)}{\Poisson_{\La_R}(E^c)} \geq -{e}^{-1}\;.
\end{align*}
since $r\log r\geq -{e}^{-1}$. Thus we obtain
\begin{equation*}
  \Ent[\chP_{R}|\Poisson_{\La_R}] \leq \Ent[\P_{\La_R}|\Poisson_{\La_R}] + \frac{1-\P(E)}{\P(E)}\Ent[\P_{\La_R}|\Poisson_{\La_R}] - \log \P(E) + \frac{1}{e} \frac{\Poisson_{\La_R}(E^c)}{\P_{\La_R}(E)}.
\end{equation*}
Using the fact that $\P_{\La_R}(E)=\P(E)\geq 1-\epsilon$, and assuming e.g.~$\epsilon \leq \hal$, it yields \eqref{eq:ent-restr}.

\paragraph{\textbf{Step 2:}} Let us now compare $\Ent[\chP_{R}|\Poisson_{\La_R}]$ and
$\Ent[\tPR|\Poisson_{\La_R}]$.

Let $\config(\La_R)$ denote the set of point configurations in $\La_R$ and $\config(\Old)$ denote the set of point configurations on $\Old = \La_{R'}$. Using the disintegration formula
\eqref{eq:ent-desint} for the map $T : \config(\Lambda_R) \to \config(\Old)$
given by the restriction $\C \mapsto \C\big|_\Old$ we have
\begin{equation}
  \label{eq:ent-est1}
  \begin{split}
    \Ent[\chP_{R}|\Poisson_{\La_R}] &= \Ent[T_{*}\chP_{R}|T_{*} \Poisson_{\La_R}] +
    \int_{\config(\Old)}\Ent\big[\chP_R(\cdot|T=\mathcal
    C)|\Poisson_{\La_R} (\cdot|T=\mathcal C)\big]d(T_{*}\chP_R)(\mathcal C)\;,\\
 \Ent[\tPR|\Poisson_{\La_R}] &= \Ent[T_{*}\tPR|T_{*} \Poisson_{\La_R}] +
    \int_{\config(\Old)}\Ent\big[\tPR(\cdot|T=\mathcal
    C)|\Poisson_{\La_R}(\cdot|T=\mathcal C)\big]d(T_{*}\tPR)(\mathcal C)\;,\\
  \end{split}
\end{equation}
Note that $T_{*}\Poisson_{\La_R}=\Poisson_{\Old}$. Further, we have by construction that $T_{*}\tPR=T_{*}\chP_R$, hence the first terms in the decomposition of the entropy of $\chP_R$
and $\tPR$ coincide, and it remains to compare the conditional entropies.

By definition of the Poisson point process, a random configuration drawn from
$\Poisson_{\La_R}(\cdot|T=\mathcal C)$ consists of $\mathcal C$ in $\Old$ plus a sample of
points in $\New =\Lambda_R \setminus \Old$ drawn from $\Poisson_{\New}$, the Poisson
process in $\New$. Denoting by $\tPR^{\mathcal C}$ the image of
$\tPR(\cdot|T=\mathcal C)$ under restriction to $\New$ we have
\begin{equation*}
 \Ent \left[\tPR(\cdot|T=\mathcal C)|\Poisson_{\La_R}(\cdot|T=\mathcal C)\right] = \Ent\left[\tPR^{\mathcal C}|\Poisson_{\New}\right]\;.
\end{equation*}

By construction, a random configuration drawn from
$\tPR(\cdot|T=\mathcal C)$ consists of $\mathcal C$ in $\Old$ plus
$n(\mathcal C)$ points placed independently uniformly in the intervals
$(p_i(\mathcal C)-\eta,p_i(\mathcal C)+\eta)$. To calculate the entropy with respect to the Poisson process in $\New$, let us introduce another disintegration concerning the number of points placed in
$\New$. For a set $I\subset\R$ we set
$$
\config(I)^{(n)}:= \{ \C\in \config(I), |\C| = n \}.
$$ 
Given a probability $\sigma$ on $\config(I)$ we consider the restrictions
\begin{align}\label{eq:restrictions}
  \sigma_n := \frac{1}{\sigma(n)}\sigma\big|_{\config(I)^{(n)}}, \quad \text{ where we denote } \sigma(n) :=\sigma\big(\config(I)^{(n)}\big)\;.
\end{align}
By definition of the Poisson point process, we have $\Poisson_{I}(n)= e^{-|I|}\frac{|I|^n}{n!}$, and the law of $\Poisson_{I, n}$ is that of $n$ independent points placed uniformly in $I$. 
Using again the disintegration formula \eqref{eq:ent-desint}, we get
\begin{align}\label{eq:ent-nb-points}
  \Ent[\sigma|\Poisson_{I}] = \sum_{n=0}^\infty \sigma(n)
\Big[\log\Big(\frac{\sigma(n)}{\Poisson_{I}(n)}\Big)+\Ent[\sigma_n|\Poisson_{I, n}]\Big]\;.
\end{align}

Recalling that $\tPR^{\mathcal C}$ is the law of $n(\mathcal C)$
independent points each placed uniformly into $n(\mathcal C)$
intervals of size $2\eta$ we find that for $n=n(\mathcal C)$
the entropy of $\left(\tPR^{\mathcal C}\right)_n$ with respect to $\Poisson_{\New, n}$ is given by the entropy of the uniform distribution on the set
\begin{align*}
  B_{n,\eta}:=\bigcup_{\sigma\in S_n}(p_{\sigma(1)}-\eta,p_{\sigma(1)}+\eta)\times\cdots\times(p_{\sigma(n)}-\eta,p_{\sigma(n)}+\eta)
\end{align*}
relative to the uniform distribution on $\New^n$, i.e.
\begin{align*}
\Ent\left[\left(\tPR^{\mathcal C}\right)_n|\Poisson_{\New, n}\right] = \log\frac{|\New|^n}{|B_{n,\eta}|} 
  = \log \frac{|\New|^n}{n!(2\eta)^n}  \;.
\end{align*}
By construction we have $\tPR^{\mathcal C}(m)=1$ for
$m=n(\mathcal C)$ and $0$ otherwise. Thus we obtain from \eqref{eq:ent-nb-points} that for $n=n(\mathcal C)$ 
\begin{align}
\label{eq:ent-est3}
  \Ent[\tPR^{\mathcal C}|\Pi_{\New}] &= \log \frac{1}{e^{-|\New|}|\New|^n/n!} + \Ent\left[\left(\tPR^{\mathcal C}\right)_n|\Poisson_{\New, n}\right]
  \\\nonumber
&= \log\frac{e^{|\New|}n!}{|\New|^n} +\log \frac{|\New|^n}{n!(2\eta)^n} = |\New|-n(\mathcal C)\log 2\eta
\\\nonumber
&= 2sR -\left(2R-\C(\Old)\right)\log 2\eta.
\end{align}
Finally, note that $\Ent[\chP_R(\cdot|T=\mathcal C)|\Poisson_{\La_R}(\cdot|T=\mathcal C)]\geq 0$. Thus combining \eqref{eq:ent-est1} and \eqref{eq:ent-est3} we finally obtain
\begin{align*}
  \Ent[\tPR|\Poisson_{\La_R}] \leq \Ent[\chP_R|\Poisson_{\La_R}] + 2sR -\log 2\eta \E_{\chP_R}\left[2R - |\C_{\Old} | \right]. 
\end{align*}
By construction (see \eqref{LeftRight}), we have 
$$
\left|2R - |\C_{\Old}| \right| \leq 2sR+sR^{1/2}  \leq 3sR.
$$
 We thus obtain
\begin{equation}
\label{eq:ent-est4}
\Ent[\tPR|\Poisson_{\La_R}] \leq \Ent[\chP_R|\Poisson_{\La_R}] + 5|\log \eta|sR.
\end{equation} 

Combining \eqref{eq:ent-restr} and \eqref{eq:ent-est4}, we obtain (for $R$ large, the constant error term in \eqref{eq:ent-restr} can be absorbed in the dominant error terms)
\begin{equation*}
\frac{1}{|\La_R|} \Ent[\tPR|\Poisson_{\La_R}] \leq \frac{1}{|\La_R|} \Ent[\P_{\La_R} | \Poisson_{\La_R}] + \frac{2 \epsilon}{|\La_R|} \Ent[\P_{\La_R} | \Poisson_{\La_R}] + \frac{5}{2}|\log \eta|s.
\end{equation*}
First, we observe that the error term $|\log \eta|s $ is of the same order as $\ErrScr$ and $s$ was chosen small enough so that $\ErrScr$ is small. On the other hand the limit defining the specific relative entropy is non-decreasing (it follows from a super-additivity argument, see e.g. \cite[Cor. 6.77]{friedli2017statistical} where $S$ there is the opposite of our  entropy), so
$$
\frac{1}{|\La_R|} \Ent[\P_{\La_R} | \Poisson_{\La_R}] + \frac{2 \epsilon}{|\La_R|} \Ent\left[\P_{\La_R} | \Poisson_{\La_R}\right] \leq \SRE(\P) + 2\epsilon \SRE(\P), 
$$
and up to working with a smaller $\epsilon'$ we may ensure that $2 \epsilon' \Ent[\P_{\La_R} | \Poisson_{\La_R}] + 5 |\log \eta| s \leq \epsilon$, hence we obtain \eqref{goodEnt}.

\subsubsection{Additional considerations}
Since \eqref{TotDiscr} holds, and since Claim \ref{claim:position} ensures that there are $sR \pm s M^{1/2} R^{1/2}$ points in $[-R, -R']$, we see that there are $R \pm s M^{1/2} R^{1/2}$ points of $\Cscr$ in $[-R,0]$ and thus \eqref{controlS} holds.

Finally, \eqref{discrpostscr} follows from combining the definition of $\TotDiscr$ in $\Old$ with the discrepancy estimates in $\New$ \eqref{discrNewA}, \eqref{discrNewB}.

\section{Displacement convexity}
\label{sec:displacement}
\subsection{Labelling on the orthant}
In this section, it is more convenient to treat the laws of random $2R-$point configurations in $\La_R$ as probability measures on $2R$-tuples. In order to identify a configuration and a $2R$-tuple, we introduce the orthant $\OrN$
\begin{equation}
\label{ON}
\OrN := \left\lbrace (\Xx_1, \dots, \Xx_{2R}) \in \left(\La_N\right)^{2R} \ | \ \Xx_1 \leq \dots \leq \Xx_{2R}\right\rbrace.
\end{equation}

\begin{itemize}
\item We denote by $\configNN$ the set of point configurations in $\Lambda_R$ with exactly $2R$ points.
\item We let $\BN$ be the Bernoulli point process with $2R$ points in $\La_R$, which is a probability measure on $\configNN$.
\item We let $\LN$ be the normalized Lebesgue measure on $\OrN$.
\end{itemize}

\begin{definition}[Label map]
We define the label map $\piN$ as
\begin{equation}
\label{def:piN}
\piN : \begin{cases} 
\configNN & \longrightarrow  \OrN \subset \left(\Lambda_R\right)^{2R} \\
\C = \sum_{i=1}^{2R} \delta_{\Xx_i}  & \mapsto (\Xx_1, \dots, \Xx_{2R}), \quad \Xx_1 \leq \dots \leq \Xx_{2R}.
\end{cases}
\end{equation}
\end{definition}

\begin{lemma}[$\piN$ is essentially bijective]
\label{lem:piNbijec}
The label map $\piN$ is well-defined and is a bijection from $\configNN$ to $\OrN$, up to a subset of measure zero (for $\BN$) in the source and a subset of measure zero (for $\LN$) in the target. 
\end{lemma}
\begin{proof}
The label map $\piN$ as in \eqref{def:piN} is well-defined and injective on the set of simple configurations, when all points are distinct. The image of this set in $\OrN$ is the set of strictly ordered $2R$-tuples $\Xx_1 < \Xx_2 < \dots < \Xx_{2R}$. It is clear that the first set has full measure in $\configNN$ (for $\BN$), and the second set has full measure in $\OrN$ (for $\LN$).
\end{proof}

The following fact is easy to check.
\begin{lemma}[Effect of $\piN$ on $\BN$]
\label{lem:pushBNpiN}
$\BN$ and $\LN$ are images of each other under $\piN$ and its inverse, i.e.~${\piN}_*\BN=\LN$ and $(\piN^{-1})_*\LN=\BN$.
\end{lemma}

\begin{lemma}[Effect of $\piN$ on the entropy]
\label{prop:piNentro}
Let $\P$ be a point process on $\La_R$ with almost surely $2R$ points, i.e.~$\P(\configNN)=1$, and let $\hP$ be its image under $\piN$, i.e.~$\hP=(\pi_R)_*\P$. Then, there exists a constant $c_R$ depending only on $R$ such that 
\begin{equation}
\label{LNversusPoisson}
\Ent[\P | \Poisson_{\La_R}] = \Ent[\hP | \LN ] + c_R\;.
\end{equation}
\end{lemma}

\begin{proof}
  Since $\P$ has almost surely $2R$ points and noting that the restriction of $\Poisson_{\La_R}$ to $\configNN$ is $\BN$, we infer from the
  desintegration formula \eqref{eq:ent-nb-points} that 
  \begin{align*}
  \Ent[\P|\Poisson_{\La_R}] &= \log \frac{e^{2R}(2R)!}{|2R|^{2R}} + \Ent[\P|\BN]
\end{align*}
  Since $\pi_R$ is essentially bijective, we have that
  $\P=(\pi^{-1}_R)_*\hP$. By Lemma \ref{lem:pushBNpiN}, $\LN=(\pi_R)_*\BN$ and
  $\BN=(\pi_R^{-1})_*\LN$. Now, the claim follows
  immediately from \eqref{eq:ent-push}.
\end{proof}

\subsection{The optimal transportation map for the quadratic cost}
\subsubsection*{Definition of the transportation map}
Let $\hPz, \hPu$ be two probability measures on $\OrN$, with finite relative entropy with respect to $\LN$. In particular, they are absolutely continuous with respect to $\LN$.

\begin{prop}[Existence of a transportation map]
There exists a map $\TN : \OrN \to \R^{2R}$ satisfying: 
\begin{itemize}
\item The push-forward of $\hPz$ by $\TN$ is $\hPu$.
\item There exists a convex function $\varphi : \OrN \to \R$ such that $\TN = \nabla \varphi$.
\end{itemize}
\end{prop}
\begin{proof}
This follows from \cite[Theorem 2.12]{villani2003topics}, because $\hPz$ and $\hPu$ are both compactly supported probability measures on $\R^{2R}$, absolutely continuous with respect to the Lebesgue measure.
\end{proof}
The first item expresses the fact that $\TN$ transports $\hPz$ onto $\hPu$. The fact that $\TN$ is the gradient of a convex function expresses the \textit{optimality} of $\TN$ for the quadratic cost. We will use the optimal character of $\TN$ only once, to argue that the relative entropy is displacement convex.

\begin{definition}[Displacement interpolation]
For any $t \in [0,1]$, we introduce the displacement interpolate $\hP^t$ as
\begin{equation}
\label{def:hPt}
\hP^t := \left((1-t) \id + t \TN \right)_* \hPz,
\end{equation}
which is consistent with the previously defined $\hPz, \hPu$ in the cases $t = 0$ and $t=1$. 
\end{definition}

\begin{definition}[Half-interpolate]
\label{defi:hip}
Let $\Xz = (\Xx_1^\z, \dots, \Xx_{2R}^\z)$ be a $2R$-tuple in $\OrN$, and let $\Xu = (\Xx_1^\u, \dots, \Xx^{\u}_{2N})$ in $\OrN$ be the image of $\Xz$ by the transportation map $\TN$. We introduce $\Xhal$ as 
\begin{equation}
\label{def:Xhal}
\Xhal := \frac{\id + \TN}{2}(\Xz) := \left( \frac{\Xx^\z_1 + \Xx^\u_1}{2}, \dots, \frac{\Xx_{2R}^\u + \Xx^\u_{2R}}{2} \right).
\end{equation}
\end{definition}

\subsubsection*{Effect of the transport on the discrepancy}
\begin{lemma}
\label{lem:transportondiscr}
Let $R > 0$, let $\Cz, \Cu$ be two point configurations with $2R$ points in $\La_R$, and let $\Chal$ be the half-interpolate of $\Cz$ and $\Cu$. For any $r$ in $[-R,R]$, we have
\begin{equation}
\label{discrpartransport}
\left| \Discr_{[-R,r]}(\Chal) \right| \leq \max \left( \left| \Discr_{[-R,r]}(\Cz)\right|, \left|\Discr_{[-R,r]}(\Cu)\right| \right)
\end{equation}
\end{lemma}
\begin{proof}
The construction of $\Chal$ clearly implies, for any $r \in [-R,R]$,
$$
\min \left(\left|\Cz_{[-R,r]}\right|, \left|\Cu_{[-R,r]}\right|\right) \leq \left|\Ch_{[-R,r]}\right| \leq \max\left(\left|\Cz_{[-R,r]} \right|, \left|\Cu_{[-R,r]}\right|\right), 
$$
and the results follows from the definition of $\Discr$.
\end{proof}

In particular, since the configurations that we construct all satisfy \eqref{discrpostscr}, this is also the case for all the configurations obtained by interpolation.

\subsubsection*{Displacement convexity of the entropy}
\begin{lemma}[Displacement convexity of the entropy]
\label{lem:disconventr}
The map 
$t \mapsto \Ent\left[\hP^t| \LN \right]$
is convex on $[0,1]$. 
\end{lemma}
\begin{proof}
This is a well-known “displacement convexity” result, which was proven in the pioneering paper \cite{MCCANN1997153}[Theorem 2.2]. 
\end{proof}

\begin{remark}
This is the only moment where we use the fact that $\TN$ is the \textit{optimal} transport map for the quadratic cost.
\end{remark}

In particular, we obtain
\begin{equation}
\label{convexentropy1}
\Ent\left[\hP^{\hal} | \LN \right] \leq  \hal \left( \Ent\left[\hP^{1} | \LN \right] + \Ent\left[\hP^{2} | \LN  \right] \right).
\end{equation}
If $\Pz_R, \Phal_R, \Pu_R$ are the push-forward of $\hPz, \hP^{\hal}, \hPu$ by $\piN^{-1}$, using Lemma \ref{prop:piNentro} we see that
\begin{equation}
\label{convexentropy2}
\Ent\left[\Phal_R | \Poisson_{\La_R} \right] \leq  \hal \left( \Ent\left[\Pz_R | \Poisson_{\La_R}  \right] + \Ent\left[\Pu_R | \Poisson_{\La_R}  \right] \right).
\end{equation}

Since $x \mapsto x\ln(x)$ is a strictly convex function on $(0, +\infty)$, the results of \cite{MCCANN1997153} imply that the inequality in \eqref{convexentropy1}, and thus in \eqref{convexentropy2}, is strict unless $\Pz_R = \Pu_R$. Inspecting the proof could possibly yield some quantitative bound, but, here, we rely instead on the energy term to get a tractable convexity inequality along the displacement interpolation.

\subsection{Convexity inequality for the energy}
 Although \cite{MCCANN1997153}[Section 3] deals with the displacement convexity of energies similar to our, the setting is different and we cannot directly apply these results. Moreover, we crucially need a quantitative convexity estimate, which is the aim of this section.

\subsubsection*{Convexity for the logarithmic interaction}
We start by stating an elementary inequality:
\begin{lemma}[Quantitative convexity for $-\log$]
Let $x,y > 0$. We have
\begin{equation}
\label{convlog}
- \log \left( \frac{x+y}{2} \right) \leq \frac{- \log x - \log y}{2} - \frac{(x-y)^2}{8(x^2 + y^2)}.
\end{equation}
\end{lemma}
\begin{proof}
The function $f : x \mapsto - \log (x)$ is convex on $(0, + \infty)$, since its second derivative is given by $f''(x) = \frac{1}{x^2} > 0$. The result follows by applying the following elementary inequality 
$$
f\left( \frac{x+y}{2} \right) \leq \frac{f(x) + f(y)}{2} - \frac{1}{8} (x-y)^2 \inf_{[x,y]} f''(c). 
$$
\end{proof}
As an immediate consequence, if $(\Xx^\z_1, \dots, \Xx^\z_{2R})$ and $(\Xx^\u_1, \dots, \Xx^\u_{2R})$ are two $2R$-tuples of points in $\OrN$, we have
\begin{multline}
\label{convexityenergyA}
\sum_{1 \leq i < j \leq 2R} \frac{ - \log (\Xx_j^\z - \Xx_i^\z) - \log (\Xx_j^\u - \Xx_i^\u) }{2} + \log \left( \frac{ \Xx_j^\z + \Xx_j^\u}{2} - \frac{\Xx_i^\z + \Xx_i^\u}{2} \right) \\
\geq \frac{1}{8} \sum_{1 \leq i < j \leq 2R} \frac{ \left( (\Xx_j^\z - \Xx_i^\z) - (\Xx_j^\u - \Xx_i^\u) \right)^2 }{ (\Xx_j^\z - \Xx_i^\z)^2 + (\Xx_j^\u - \Xx_i^\u)^2}.
\end{multline}

\begin{definition}[Gaps, counted from the left] 
Let $\X = (\Xx_1, \dots, \Xx_{2R})$ be a $2R$-tuple in $\OrN$. We define the gaps of $\X$, “counted from the left” as
\begin{equation}
\label{def:gaps2}
\hGap_{i} := \Xx_{i+1} - \Xx_{i} \ (1 \leq i \leq 2R-1)
\end{equation}
We denote by $\hGap_i(\X)$ the $i$-th such gap of a given $\X \in \OrN$.
\end{definition}

\begin{definition}
Let $\X = (\Xx_1, \dots, \Xx_{2R})$ be a $2R$-tuple of points in $\OrN$, and let $\C = \pi_{R}^{-1}(\X)$. We introduce the notation
\begin{equation}
\label{intX}
\Int[\X] :=  \sum_{1 \leq i < j \leq 2R} - \log|\Xx_i - \Xx_j| = \iint_{x < y} - \log |x- y| d\C(x) d\C(y). 
\end{equation}
\end{definition}

\begin{prop}[Quantitative convexity along displacement]
\label{prop:strictXhal}
Let $\Xz, \Xu, \Xhal$ be as in Definition \ref{defi:hip}. We have:
\begin{equation}
\label{strictXhal}
\Int[\Xhal] \leq \frac{\Int[\Xz] + \Int[\Xu]}{2} - \frac{1}{8} \sum_{i = 1}^{2R - 1} \frac{\left|\hGap_i(\Xz) - \hGap_i(\Xu) \right|^2}{\left(\hGap_i(\Xz)\right)^2 +  \left(\hGap_i(\Xu)\right)^2}.
\end{equation}
\end{prop}
\begin{proof}
It is a straightforward consequence of \eqref{convexityenergyA}, by writing
\begin{equation*}
\sum_{1 \leq i < j \leq 2R} \frac{ \left( \left( (\Xx_j^\z - \Xx_i^\z \right) - \left(\Xx_j^\u - \Xx_i^\u \right) \right)^2 }{ \left(\Xx_j^\z - \Xx_i^\z \right)^2 + \left(\Xx_j^\u - \Xx_i^\u)\right)^2}
\geq \sum_{i=1}^{2R-1} \frac{ \left( \left(\Xx_{i+1}^\z - \Xx_i^\z \right) - \left(\Xx_{i+1}^\u - \Xx_i^\u\right) \right)^2 }{ \left(\Xx_{i+1}^\z - \Xx_i^\z\right)^2 + \left(\Xx_{i+1}^\u - \Xx_i^\u\right)^2},
\end{equation*}
and using the notation \eqref{def:gaps2}.
\end{proof}

\begin{definition}[The background contribution]
For $R > 0$, we introduce the “background potential” as 
\begin{equation}
\label{def:VR}
\V_R(t) := \int_{-R}^R \log |t-s| ds.
\end{equation}
\end{definition}
An explicit computation yields
  \begin{align*}
    \V_R(t)&=\int_{-R}^t\log(t-s)ds + \int_t^R\log(s-t)ds
         =\Big[(R+t)\log(R+t)-(R+t)+(R-t)\log(R-t)-(R-t)\Big]\;.
  \end{align*}
and thus
\begin{equation}
\label{VRtsec}
\V'_R(t)=\log(R+t)+\log(R-t), \quad \V''_R(t)  =   \frac{1}{R+t} + \frac{1}{R-t}.
\end{equation}
As a straightforward consequence, we obtain:
\begin{lemma}
\label{lem:BFconvex}
$\V_R$ is a convex function on $\La_R$.
\end{lemma}

\begin{lemma}[Intrinsic energy, logarithmic interaction, background field]
Let $\C$ be in $\configNN$, and let $\X = \piN(\C)$. We have
\begin{equation}
\label{defWintIntVN}
\Wint(\C; \La_R) = 2 \Int[\X] + 2 \sum_{i=1}^{2R} \V_R(\Xx_i) + \const_R,
\end{equation}
where $\Int$ is as in \eqref{intX}, $\V_R$ is as above, and $\const_R$ is a constant depending only on $R$.
\end{lemma}
\begin{proof}
It follows from the definition \eqref{def:Wint} of $\Wint$, by expanding the quadratic term $(d\C - dx)(d\C - dy)$,
\begin{multline*}
\Wint(\C; \La_R) = \iint_{\La_R \times \La_R \setminus \diamond} - \log |x-y| (d\C -dx) (d\C - dy)  = 2 \iint_{x < y} - \log |x-y| d\C(x) d\C(y) \\ + 2 \iint_{\La_R \times \La_R \setminus \diamond} \log |x-y| d\C(x) dy 
+ \iint_{\La_R \times \La_R \setminus \diamond} - \log |x-y| dx dy,
\end{multline*}
and using definition \eqref{def:VR}.
\end{proof}

\subsubsection*{The convexity inequality}
\begin{definition}[Gain and background field contribution]
Let $\Cz, \Cu$ be in $\configNN$, $\Xz, \Xu$ be their image by $\piN$, let $\Xhal$ be as in \eqref{def:Xhal}.
\begin{itemize}
\item We define the “gain” term $\GGain(\Cz, \Cu; \La_R)$ as:
\begin{equation}
\label{def:GGain}
\GGain(\Cz, \Cu; \La_R) := \sum_{i = 1}^{2R - 1} \frac{\left|\hGap_i(\Xz) - \hGap_i(\Xu) \right|^2}{\left(\hGap_i(\Xz)\right)^2 +  \left(\hGap_i(\Xu)\right)^2}.
\end{equation}
\item We define the “background field contribution” term $\BF(\Cz, \Cu; \La_R)$ as
\begin{equation}
\label{def:BF}
\BF(\Cz, \Cu; \La_R) := 2 \sum_{i=1}^{2R} \lp  \V_R(\Xx_i^\hh) - \hal\left( \V_R(\Xx_i^\z) + \V_R(\Xx_i^\u) \right) \rp.
\end{equation}
\end{itemize}
\end{definition}

\begin{prop}
\label{prop:Wintconvex}
Let $\Cz, \Cu$ be in $\configNN$, we have:
\begin{equation}
\label{Wintconvex}
\Wint(\Chal;\Lambda_R) \leq \hal\left(\Wint(\Cz;\Lambda_R) + \Wint(\Cu;\Lambda_R)\right) - \frac{1}{4} \GGain(\Cz, \Cu; \La_R).
\end{equation}
\end{prop}
\begin{proof}
We combine \eqref{strictXhal} and \eqref{defWintIntVN}, and we use the notation introduced in \eqref{def:GGain}, \eqref{def:BF}. We obtain
$$
\Wint(\Chal;\Lambda_R) \leq \hal\left(\Wint(\Cz;\Lambda_R) + \Wint(\Cu;\Lambda_R)\right) - \frac{1}{4} \GGain(\Cz, \Cu; \La_R) + \BF(\Cz, \Cu; \La_R).
$$
Since $\V_R$ has been observed to be a convex function on $\La_R$ (see Lemma \ref{lem:BFconvex}), from the definition of $\BF$ in \eqref{def:BF} we see readily that $\BF(\Cz, \Cu; \La_R) \leq 0$, so in particular\footnote{In fact, the contribution of the backgroung potential can be shown to be negligible with respect to $R$.} \eqref{Wintconvex} holds.
\end{proof}

\section{The interpolate process}
\label{sec:interpolateprocess}
Let $\Pz, \Pu$ be two minimisers of $\fbeta$, and assume that $\Pz \neq \Pu$. Let $\g$ be the “gain” given by Proposition \ref{prop:gaps2}, i.e.\ $1/R$ times the right hand side of \eqref{EPigaps}.

\subsection{Definition of the interpolate process}
\begin{definition}
For any $s >0$ and for $R$ large enough (depending on $s, \Pz, \Pu$), we apply the “large box approximation” of Proposition \ref{prop:largeboxapproximation} to $\Pz, \Pu$ with $\epsilon = \frac{\g}{100}$. We let $\tPzR, \tPuR$ be the processes on $\La_R$ obtained this way. Further, we let $\hPz_R=(\pi_R)_*\tPzR$, $\hPu_R=(\pi_R)_*\tPuR$ be the correponding measures on $\OrN$ obtained by pushforward via the labeling map $\pi_R$. We let $\hat Q_R$ be an optimal coupling of $\hPz_R$ and $\hPu_R$ given by the optimal transport map $\Trans_R$ and let $\hat \Phal_R$ be the half-interpolate measure on $\OrN$ obtained from displacement interpolation. Finally, we let $\Phal_R=(\pi_R^{-1})_*\tilde \Phal_R$ be the corresponding half-interpolate process on $\La_R$ and let $Q_R=(\pi_R^{-1}\times\pi_R^{-1})_*\tilde Q_R$ be the corresponding coupling of  $\tPzR, \tPuR$.
\end{definition}

\subsection{The energy gain is proportional to the size of the segment}
\def \Sz{S^\z}
\def \Su{S^\u}
\begin{lemma}[The energy gain]
\label{lem:gainisextensive}
Taking $R$ large enough, we have
\begin{equation}
\label{gainisextensive}
\E_{\Q_R} \left[\GGain( \cdot,  \cdot ; \La_R) \right] \geq \hal \g R.
\end{equation}
\end{lemma}
\begin{proof}
  Let $\Cscrz, \Cscru$ be in the support of $\tPzR, \tPuR$. Let
  $\Sz, \Su$ be the quantities, as in \eqref{def:SC}, relating the two
  ways of enumerating points, i.e. such that
$$
x_0(\Cscrz) = \Xx^\z_{R + \Sz}, \quad x_0(\Cscru) = \Xx^\u_{R + \Su},
$$
which means $\Cscrz$ (resp. $\Cscru$) has $R + \Sz$ (resp. $R + \Su$) points in $[-R,0]$. 

From \eqref{controlS} we know that (up to choosing $s$ small enough with respect to $M$) $\Sz, \Su$ are bounded by $\hal R^{1/2}$, so $S = \Su - \Sz$ is bounded by $R^{1/2}$. We may write, using the definition \eqref{def:GGain} and switching indices:
$$
\GGain(\Cscrz, \Cscru; \La_R) \geq \sum_{i=-R/2}^{R/2} \frac{\left|\Gap_i(\Cscrz) - \Gap_{i+S}(\Cscru) \right|^2}{\left(\Gap_i(\Cscrz)\right)^2 +  \left(\Gap_{i+S}(\Cscru)\right)^2}.
$$
Note that the right hand side only depends on the restrictions of
  $\Cscrz,\Cscru$ to $\Old=\La_{R'}$. Let us denote by $\bar \Q_R$ the
  image of the coupling $\Q_R$ under restriction to $\Old$. Since the
  screening procedure does not change the configurations in $\Old$ by
  construction, $\bar \Q_R$ is a coupling of the restrictions to
  $\Old$ of the stationary processes $\Pz$ and $\Pu$ conditioned on an
  event of probability $\geq 1 - \frac{\g}{100}$.  
$$
\E_{\Q_R} \left[\GGain( \cdot,  \cdot ; \La_R) \right] \geq \E \left[ \sum_{i=-R/2}^{R/2} \frac{\left|\Gap_i(\Cz) - \Gap_{i+S}(\Cu) \right|^2}{\left(\Gap_i(\Cz)\right)^2 +  \left(\Gap_{i+S}(\Cu)\right)^2} \right]\;,
$$
where $(\Cz,\Cu)$ is distributed according to $\bar \Q_R$.
In view of Corollary \ref{coro:gaps}, the last expression is bounded below by $\hal\g R$, which yields the result.
\end{proof}

\subsection{Energy of the interpolate process}
\begin{prop}
For $s$ small enough (depending on $\Pz, \Pu$), for $R$ large enough (depending on $\Pz, \Pu, s$) we have
\begin{equation}
\label{EPhalWint}
\frac{1}{|\La_R|} \E_{\Phal} \left[\Wint(\Chal; \La_R)\right] \leq \hal \left( \Welec(\Pz) + \Welec(\Pu) \right) - \frac{\g}{4} 
\end{equation}
\end{prop}
\begin{proof}
We combine 
\begin{itemize}
\item \eqref{goodEner}, from the approximation procedure, that controls the energies $\E_{\Pz}[\Wint], \E_{\Pu}[\Wint]$ in terms of $\Welec(\Pz)$, $\Welec(\Pu)$.
\item \eqref{Wintconvex}, from the transportation argument, that expresses $\E_{\Phal}[\Wint]$ in terms of $\E_{\Pz}[\Wint], \E_{\Pu}[\Wint]$ and the “gain” term.
\item Lemma \ref{lem:gainisextensive}, saying that the “gain” is proportional to the volume.
\end{itemize}
\end{proof}

\section{Conclusion: proof of Theorem \ref{theo:main}}
\label{sec:conclusionproof}
\subsection{Constructing a better candidate}
We now turn the interpolate process, which is supported in a large segment $[-R,R]$, into a stationary point process on the whole real line. We proceed in two steps: first we paste independent copies of the interpolate process, and then we average over translations in the original segment. This construction was used in \cite{leble2016logarithmic} for similar purposes.

\subsubsection*{Pasting copies on the line}
Let $\{K_i : = \La_R - 2Ri \}_{i \in \Z}$ be a tiling of $\R$ by translates of the interval $\La_R$, and let $\{\Chal_i\}_{i \in \Z}$ be independent random variables identically distributed according to  $\Phal$. We let $\Paste_R$ be the map
$$
\Paste_R := \{ \C_i \}_{i \in \Z} \mapsto \sum_{i \in \Z} (\C_i - 2iR),
$$
that creates a configuration on $\R$ from a collection of configurations on $\La_R$ by copying one on each interval $K_i$. We let $\Ptot$ be the push-forward of $\{\Chal_i\}_{i \in \Z}$ by $\Paste_R$.

\subsubsection*{The stationary candidate: averaging over an interval}
We let $\Pav$ be the law of the point process defined by “averaging $\Ptot$ over translations in $\La_R$”. More precisely, we let $\Pav$ be the law of the point process defined by duality as follows: for any bounded measurable test function $F$, 
\begin{equation}
\label{def:Pav}
\E_{\Pav} [F] := \E_{\Ptot} \left[  \frac{1}{|\La_R|} \int_{-R}^R F(\C - t) dt \right].
\end{equation}
Now, by construction, $\Pav$ is a stationary process of intensity $1$, and there are always $2L \pm 4R$ points in any interval of length $2L$.

\subsection{Claims about the new process and conclusion}
\begin{prop}[The new process is a better candidate]
The process $\Pav$ satisfies
\begin{equation}
\label{fbetaPav} \fbeta(\Pav) < \hal \lp \fbeta(\Pz) + \fbeta(\Pu) \rp.
\end{equation}
\end{prop}
\begin{proof}
Recall that $\fbeta(\Pav) := \beta \Welec(\Pav) + \SRE(\Pav)$, we evaluate each term below.

\subsubsection*{The specific relative entropy}
\begin{claim}[The specific relative entropy of $\Ptot$]
\label{claim:SREPtot}
We have $\displaystyle{\SRE(\Pav) \leq \hal \left( \SRE(\Pz) + \SRE(\Pu)  \right) + \frac{\g}{50}}$.
\end{claim}
\begin{proof}
By construction, the process $\Ptot$ is made of independent copies of $\Phal$ on $\La_R$, so we have
\begin{equation*}
\SRE(\Pav) = \lim_{M \to \infty} \frac{1}{|\La_M|} \Ent\left [\Pav_{\La_M} | \Poisson_{\La_M} \right] = \frac{1}{|\La_R|} \Ent\left[ \Phal | \Poisson_{\La_R} \right].
\end{equation*}
Using \eqref{convexentropy2} and \eqref{goodEnt}  (with $\epsilon = \frac{\g}{100}$) yields the claim.
\end{proof}

\subsubsection*{The intrinsic energy}
\begin{claim}[The intrinsic energy of $\Pav$]
\label{claim:intrinsicPtot}
We have $\displaystyle{\Wintg(\Pav) \leq \hal \left( \Welec(\Pz) + \Welec(\Pu) \right) - \frac{\g}{10}}$.
\end{claim}
\begin{proof}
Since $\Wintg$ is defined in \eqref{def:Wintglob} as a $\liminf$, it is enough to show that 
$$
\liminf_{M \to \infty} \frac{1}{|\La_{MR}|} \E_{\Phal} \left[ \Wint\left(\C ; \La_{MR} \right)  \right] \leq  \hal \left( \Welec(\Pz) + \Welec(\Pu) \right) - \frac{\g}{10}.
$$
By the definition \eqref{def:Pav}, $\Pav$ is a uniform mixture of processes that we will denote by $\Ptotz$ for $z \in \La_R$, where $\Ptotz$ is the process $\Ptot$ translated by $z$. For $z \in [-R,R]$ and $M \geq 10$ fixed, a configuration of $\Ptotz$ in $\La_{MR}$ consists of:
\begin{itemize}
\item $M-2$ “full” configurations $\C_1, \dots, \C_{M-2}$ obtained from independent copies of $\Phal$, and supported in the intervals $\La_R - z - 2Ri$ for $i \approx - M/2 \dots M/2$,
\item Two “partial” configurations $\Cleft$ and $\Cright$, where $\Cleft$ is supported in $[-MR, -MR + R - z]$ and $\Cright$ in $[MR - R - z, MR]$.
\end{itemize}
When we compute the interaction energy $\Wint(\C ; \La_{MR})$, we obtain:
\begin{equation}
\label{ABCDEFG}
\Wint(\C ; \La_{MR}) = A + B + C + E + F + G, 
\end{equation}
\begin{itemize}
\item $A$ is the sum of the interactions of a configuration $\C_i$ with itself, for $i = 1, \dots, M-2$.
\item $B$ is the interaction between $\Cleft$ and $\Cright$.
\item $C$ is the interaction between $\Cleft$ and $\C_1$; the interaction between $\Cright$ and $\C_{M-2}$.
\item $D$ is the interaction of $\Cleft$ with itself, and the interaction of $\Cright$ with itself.
\item $E$ is the sum of the interactions between $\C_i$ and $\C_{i+k}$ for $i = 1, \dots M-2$ and $2 \leq k \leq M-2-i$, i.e. between \textit{non-neighboring} “full” configurations.
\item $F$ is all the interactions between two neighboring “full” configurations $\C_i, \C_{i+1}$.
\item $G$ is the sum of the interactions between $\Cleft$ and all non-neighboring “full” configurations ($\C_i$ for $i = 2, \dots M-2$), and the interactions between $\Cright$ and all non-neighboring “full” configurations ($\C_i$ for $i = 1, \dots, M-3$).
\end{itemize}

\textbf{The term $A$ (self-interaction of full configurations).}
Taking the expectation under $\Ptotz$, in view of \eqref{EPhalWint}, we have
$$
\frac{1}{M |\La_R|} \E_{\Ptotz}(A) = \frac{M-2}{M}  \frac{1}{|\La_R|} \E_{\Phal} \left[ \Wint\left(\C ; \La_{R} \right) \right] \leq \hal \left( \Welec(\Pz) + \Welec(\Pu) \right) - \frac{\g}{4} + o_M(1). 
$$

It remains to study all the other terms and to show that they are either negligible with respect to $M$ as $M \to \infty$ or only yield a perturbation of order $MR$, that can be made arbitrarily small through the choice of $s$. 

\textbf{The terms $B, C, D$.}
We may already observe that the interaction between $\Cleft$ and $\Cright$ is bounded by $O(R^2 \log M)$, and is thus $o(M)$. So are the interactions between $\Cleft$ and $\C_1$ or between $\Cright$ and $\C_{M-2}$. We may also bound the self-interaction of $\Cleft$, $\Cright$ by a quantity independent of $M$. We thus have
$$
\frac{1}{M |\La_R|} \E_{\Ptotz}(B + C+ D) = o_M(1).
$$

\textbf{A priori bound on fluctuations.}
To control the other “pairwise” interactions we rely on the following bound expressing fluctuations in terms of discrepancies.
\begin{lemma}[Controlling fluctuations via discrepancies]
\label{lem:contrfluctgab}
Let $[a,b]$ be an interval of $\R$, let $g$ be a $C^1$ function on $[a,b]$ and let $\C$ be a point configuration on $[a,b]$. We have
\begin{equation}
\label{fluctgab}
\int_a^b g(x) (d\C - dx) \preceq \sum_{k=a}^b \|g'\|_{\infty} \left( \left|\Discr_{[a,k]}(\C)\right| + \left|\Discr_{[k,k+1]}(\C)\right| + 1 \right) + \|g\|_{\infty} \left| \Discr_{[a,b]}(\C) \right|.
\end{equation}
\end{lemma}
In particular, if $[a,b] = \La_R$ and $\C$ is a configuration with $2R$ points in $\La_R$, the last term in the right-hand side of \eqref{fluctgab} vanishes.
\begin{proof}
This follows from splitting $[a,b]$ into intervals of length $1$, using a Taylor's expansion of $g$ on each interval, and using a summation by part. We refer e.g. to \cite{Leble:2018aa}[Prop. 1.6].
\end{proof}
We use the notation $\tD_{R,k}(\C)$ for the summand in \eqref{fluctgab},
\begin{equation}
\label{def:tDRk}
\tD_{R,k}(\C) := \left| \Discr_{[-R,k]}(\C) \right| + \left|\Discr_{[k, k+1]}(\C)\right| + 1.
\end{equation}

Since the double integral defining $\Wint$ (see \eqref{def:Wint}) involves the “fluctuation” terms $(d\C - dx)(d\C - dy)$,  using \eqref{fluctgab} we can derive the following control on the interaction between two configurations living in two non-neighboring copies of $\La_R$:
\begin{prop}
\label{prop:apriori}
Let $\C^a, \C^b$ be two configurations with $2R$ points supported on the intervals $\La_R^a := \La_R - 2Ra$, $\La_R^b := \La_R - 2Rb$ respectively, with $|a-b| \geq 2$. Then
\begin{multline}
\label{interactiondiscrepanci}
\iint_{\La_R^a \times \La_R^b} - \log |x -y| (d\C^a(x) - dx) (d\C^b(y) - dy) \\
\preceq \frac{1}{|a-b|^2 R^2} \sum_{k=0}^{2R} \sum_{j=0}^{2R} \tD_{R, k}(\C^a +2Ra) \tD_{R,j}(\C^b + 2Rb),
\end{multline}
where $\tD_{R,k}(\C)$ is as in \eqref{def:tDRk}.
\end{prop}
\begin{proof}[Proof of Proposition \ref{prop:apriori}]
We apply Lemma \ref{fluctgab} twice (once for each variable) to $h(x,y) = - \log|x-y|$, bounding the second derivative of $h$, for $x$ in $\La_R^a$ and $y$ in $\La_R^b$, by $\frac{1}{|a-b|^2 R^2}$.
\end{proof}

\textbf{The term $E$.}
Combining the result of Proposition \ref{prop:apriori} with the discrepancy bounds \eqref{discrpostscr} (which are still valid for the interpolate configurations, as observed in Lemma \ref{lem:transportondiscr}), we may bound the interaction between $\C_i$ and $\C_{i+k}$, for $k \geq 2$, by 
$$
\iint_{\La_R^i \times \La_R^{i+k}} - \log |x -y| (d\C_i(x) - dx) (d\C_{i+k}(y) - dy)
\preceq  \frac{1}{k^2 R^2} \left(s^2 R^{3/2}\right)^2.
$$
For a given $i$, we thus have
$$
\sum_{k \geq 2} \iint_{\La_R^i \times \La_R^{i+k}} - \log |x -y| (d\C_i(x) - dx) (d\C_{i+k}(y) - dy) \preceq s^4 R,
$$
and summing again over $i = 1, \dots, M-2$, we bound the term $E$ in \eqref{ABCDEFG} by
$$
E \preceq s^4 M R,
$$
which is a contribution of order $MR$ that can be made arbitrarily small by taking $s$ small.

The remaining interactions, between all neighbors $\C_i,\C_{i+1}$; and between $\Cleft, \Cright$ and the $\C_i$'s, also yield arbitrarily small contributions. The argument is similar: we use Lemma \ref{lem:contrfluctgab} and the discrepancy estimates, with two small modifications.

\textbf{The term $F$.} To treat neighbors, here is a sketch of the argument: take two configurations $\C^{l}$ (on the left) and $\C^{r}$ (on the right) living in $[-R, 0]$ and $[0,R]$ respectively. In view of Lemma \ref{lem:contrfluctgab}, we write their interaction as 
\begin{multline}
\label{interaction}
\Interaction = \iint_{[-R, 0] \times [0, R]} - \log |x-y| (d\C^{l} - dx) (d\C^{r} - dy) 
\\
\preceq \sum_{i,j=1}^R \frac{1}{(i+j)^2} |\Discr_{[0, -i]}(\C^{l}) | \cdot |\Discr_{[0,j]}(\C^{r})|,
\end{multline}
where we have only kept the leading order discrepancy in \eqref{fluctgab}. The term $\frac{1}{(i+j)^2}$ comes from the fact that the second derivative $\partial_{xy} - \log|x-y|$ is controlled, for $x$ in $[-i-1, -i]$ and $y$ in $[j, j+1]$, by $\frac{1}{(i+j)^2}$. We are looking for a bound of the type
$\Interaction \preceq o_s(1) R$, since then we have $O(M)$ pairs of neighbors, each yielding a contribution $o_s(1) R$, so the term $F$ would be of order $o_s(1) MR$, which is enough for our purposes.

Note that, to simplify, when writing \eqref{interaction} we have assumed that the configurations were separated by a distance $1$ (there is no $i,j=0$ term), in reality the contribution of the terms at distance $\leq 1$ is bounded by $O(1)$ and we can forget about them. They key point is to use the fact that the second moment of the discrepancy in a segment is very small compared to the size of the segment, as expressed by Lemma \ref{lem:variancesublinear}. In particular, we may write
\begin{equation}
\label{discrisverysmall}
|\Discr_{[0, -i]}(\C^{l}) | \cdot |\Discr_{[0,j]}(\C^{r})| \ll \sqrt{i} \sqrt{j}.
\end{equation}
Then, we estimate the double sum of \eqref{interaction} as follows:
\begin{multline*}
\sum_{i,j=1}^R \frac{1}{(i+j)^2} |\Discr_{[0, -i]}(\C^{l}) | \cdot |\Discr_{[0,j]}(\C^{r})| \preceq \sum_{i=1}^R \sum_{j = i}^R \frac{1}{(i+j)^2} |\Discr_{[0, -i}(\C^{l}) | \cdot |\Discr_{[0,j]}(\C^{r})|
\\
 \ll  \sum_{i=1}^R \sum_{j = i}^R \frac{1}{(i+j)^2} \sqrt{i} \sqrt{j} 
\preceq \sum_{i=1}^R \sum_{j = i}^{2i} \frac{i}{i^2} + \sum_{i=1}^R \sqrt{i} \sum_{j = 2i}^{R} \frac{1}{j^2} \sqrt{j},
\end{multline*}
and thus a direct computation yields $\Interaction \ll R.$

\textbf{The term $G$.} If we apply Lemma \ref{lem:contrfluctgab} to $\Cleft$ (or $\Cright$), the boundary term in \eqref{fluctgab} is not zero, but bounded by $2R \|g \|_{\infty}$. In particular, when estimating the interaction between $\Cleft$ and $\C_i$, for $i \geq 2$, we obtain a term similar to \eqref{interactiondiscrepanci} above, plus a boundary term:
\begin{multline*}
\iint - \log |x -y| (d\Cleft(x) - dx) (d\C_i(y) - dy) 
\preceq \frac{1}{i^2 R^2} \left( \sum_{k=0}^{2R} \tD_{R, k}(\Cleft)  \right) \left( \sum_{j=0}^{2R} \tD_{R,j}(\C_i)\right)\\
+ \frac{1}{i R} 2R \sum_{j=0}^{2R} \tD_{R,j}(\C_i).
\end{multline*}
Using the discrepancy bound $\sum_{j=0}^{2R} \tD_{R,j}(\C_i) \preceq R^{3/2}$, the new term in the right-hand side is controlled by $\frac{1}{i} R^{3/2}$, and the sum of these terms over $i = 2, \dots, M-2$ is thus bounded by
$$
R^{3/2} \sum_{i=2}^{M-2} \frac{1}{i}  \preceq R^{3/2} \log M,
$$
which is negligible with respect to $M$.
\end{proof}

\subsubsection*{The electric energy}
\begin{claim}[The electric energy of $\Pav$]
\label{claim:electricPtot}
We have $\Welec(\Pav) < \hal \left( \Welec(\Pz) + \Welec(\Pu) \right) - \frac{\g}{10}$.
\end{claim}
\begin{proof}
It follows from the previous Claim and the “electric-intrinsic” inequality \eqref{welecwintg}, which applies here because, by construction, the discrepancy in any interval is bounded by $4R$ (see the remark following immediately \eqref{def:Pav}).
\end{proof}

\end{proof}

\subsubsection*{Conclusion of the proof}
Starting from the assumption that there are two distinct minimisers $\Pz, \Pu$ of $\fbeta$, we have constructed a stationary point process $\Pav$ which satisfies 
$$
\fbeta(\Pav) < \hal \lp \fbeta(\Pz) + \fbeta(\Pu) \rp = \min \fbeta,
$$
which is absurd. Hence, the minimiser of $\fbeta$ is unique, which proves the main theorem.

\appendix
\section{Miscellaneous proofs}
\label{sec:annex}
\subsection{Proof of Lemma \ref{lem:variancesublinear}}
\label{sec:proofvariancesublinear}
\begin{proof}
We follow the argument developed in \cite[Section 8.5]{leble2017large}, with parameters $\dd = 1, \ss = 0, \kk=1, \gamma = 0$, but the proof below is self-contained. In the sequel, equation numbers with a \textbf{bold typeface} refer to the corresponding equations in that paper.

We obtain (we could e.g. use \eqref{EnergyStatCase}) with $\eta_0 = \frac{1}{4}$ (the precise value of $\eta_0$ is, in fact, irrelevant):
\begin{equation}
\label{EPL1}
\frac{1}{2\pi} \E_{\P} \left[ \int_{\Lambda_1 \times \R} |\El_{\eta_0}|^2 \right] \leq \Welec(P) + C.
\end{equation}

For any $T > 0$, we let $\HRT$ be the rectangle $\HRT := \La_R \times [-T, T]$. Let us emphasize that here, in contrast to \textbf{(8.5)}, we do not yet fix $T$ with respect to $R$. The integration by parts as in \textbf{(8.6)} still holds, and we get
\begin{equation}
\label{pHRT}
\int_{\partial \HRT} \El_{\eta_0} \cdot \vec{\nu} =- 2\pi \left( \Discr_{\La_R}(\C) + r_{\eta_0} \right),
\end{equation}
where $r_{\eta_0}$ is an error term bounded by the number of points in a $\eta_0$-neighborhood of $\{-R, +R\}$. It is easy to see that $\E_{\P} \left[r_{\eta_0}^2\right]$ is bounded by a constant independent of $R$, and thus, since we are aiming for a $o(R)$ bound, this term is negligible - for simplicity we will forget it.

The main improvement on the existing proof is to observe that the choice $T \in (R, 2R)$ as in \textbf{(8.5)} is valid, but slightly sub-optimal. We replace it by the following claim.
\begin{claim}
\label{claim:gutentail}
There exists $f : [0, + \infty) \to [0, + \infty)$ satisfying 
\begin{equation}
\label{fx}
\lim_{x \to \infty} f(x) = + \infty, \quad \lim_{x \to \infty} \frac{f(x)}{x} = 0
\end{equation}
and such that
\begin{equation}
\label{gutentail}
\lim_{x \to \infty} x \E_{\P} \left[ \int_{\La_1 \times \{-f(x), f(x)\}} | \El_{\eta_0} |^2 \right] = 0.
\end{equation}
\end{claim}
\begin{proof}
Let $\tail(x)$ be the quantity
\begin{equation}
\label{def:tail}
\tail(x) = \E_{\P} \left[ \int_{\La_1 \times (\R \backslash (-x, x))} |\El_{\eta_0}|^2 \right].
\end{equation}
The map $x \mapsto \tail(x)$ is continuous, positive, non-increasing, with $\lim_{x \to + \infty} \tail(x) = 0$. So, introducing the map
$$
u \mapsto \frac{u}{\sqrt{\tail(u)}}, 
$$
it is continuous on $[0, + \infty)$, increasing, is equal to $0$ at $0$ and tends to $+ \infty$ at $+ \infty$. Thus, for any given $x > 0$ there exists (a unique) $u > 0$ such that
\begin{equation}
\label{utailueqx}
\frac{u}{\sqrt{\tail(u)}} = x.
\end{equation}
Now, by a mean value argument, we may find $v \in [u, 2u]$ such that
$$
\E_{\P} \left[ \int_{\La_1 \times \{-v, v\}} | \El_{\eta_0} |^2 \right] \leq \frac{1}{u} \tail(u).
$$
We define $f(x)$ as the smallest such real number $v$. It is easy to check that the properties \eqref{fx} are satisfied, and moreover we have
$$
x \E_{\P} \left[ \int_{\La_1 \times \{-f(x), f(x)\}} | \El_{\eta_0} |^2 \right] \leq \frac{x}{u} \tail(u) = \sqrt{\tail(u)},
$$
where we have used \eqref{utailueqx}. This quantity (seen as depending on $x$) tends to $0$ as $x \to \infty$, which proves \eqref{gutentail}.
\end{proof}

We now take $T = f(R)$ in \eqref{pHRT}, and we apply the Cauchy-Schwarz inequality to the left-hand side. Since we are dealing with two (slightly) different lengths $R, f(R)$, it is important to be more precise than in \textbf{(8.8)} and to split $\partial H_{R, f(R)}$ as
$$
\partial \HRT = \La_R \times \{-f(R), f(R)\} \cup \{-R, R\} \times [-f(R),f(R)].
$$
We obtain
$$
\left(\int_{\partial H_{R, f(R)}} \El_{\eta_0} \cdot \vec{\nu}\right)^2 \preceq R \int_{\La_R \times \{-f(R), f(R)\}} |\El_{\eta_0}|^2 + f(R) \int_{\{-R, R\} \times [-f(R),f(R)]} |\El_{\eta_0}|^2.
$$

By stationarity, we have, in view of \eqref{EPL1}
$$
\frac{1}{2\pi} \E_{\P} \left[ \int_{\{-R, R\} \times [-f(R),f(R)]} |E_{\eta_0}|^2 \right] = \frac{1}{2\pi} \E_{P} \left[ \int_{\La_1 \times [-f(R),f(R)]}  |E_{\eta_0}|^2 \right]  \leq \Welec(P) + C,
$$ 
and also 
\begin{equation}
\label{La1f(R)}
\E_{\P} \left[ \int_{\La_R \times \{-f(R), f(R)\}} |E_{\eta_0}|^2 \right] = R \E_{\P} \left[ \int_{\La_1 \times \{-f(R), f(R)\}} |E_{\eta_0}|^2 \right].
\end{equation}
According to \eqref{fx}, $f(R) = o(R)$, and from \eqref{gutentail} we see that the right-hand side of \eqref{La1f(R)} is $o_R(1)$. We thus obtain
\begin{equation*}
\E_{\P} \left[ \left(\int_{\partial H_{R, f(R)}} E_{\eta_0} \cdot \vec{\nu} \right)^2 \right] = R o_R(1) + o(R) \left(\Welec(P) + C\right), 
\end{equation*} 
and thus, in view of \eqref{pHRT}, we obtain $\E_{\P} \left[ \Discr^2_R(\C) \right] = o(R)$, which proves \eqref{discrPfiniB}.
\end{proof}

\subsection{Proof of Lemma \ref{lem:gapL2}}
\label{sec:proofGapL2}
\begin{proof}[Proof of Lemma \ref{lem:gapL2}]
The argument is essentially a re-interpretation of the discrepancy controls as given e.g. in \cite[Lemma 2.2]{petrache2017next}.

It is enough to show that
$$
 \sum_{i=-\frac{R}{2}}^{\frac{R}{2}} \1_{\Gap_i \geq 10} \left(\Gap_i \right)^2 \preceq \int_{[-R, R] \times \R} |\El_{\eta}|^2.
$$

Let $i$ such that $\Gap_i \geq 10$. Let $m = \frac{x_i + x_{i+1}}{2}$ and for $\ell > 0$ let $H_\ell$ be the rectangle $[m-\ell, m+\ell] \times [-\ell, \ell]$. By a mean value argument, we can find $\ell \in \left[\frac{1}{4} \Gap_i, \frac{1}{2} \Gap_i\right]$ such that
\begin{equation}
\label{mvgapL2}
\int_{\partial H_{\ell}} |\El_{\eta}|^2 \preceq \frac{1}{\Gap_i} \int_{[x_i, x_{i+1}] \times \R}  |\El_{\eta}|^2.
\end{equation}

On the other hand, using \eqref{compa2} and an integration by parts, we have:
$$
\int_{\partial H_{\ell}} \El_{\eta} \cdot \vec{n} = + \int_{[m-\ell, m+\ell] \times [-\ell, \ell]} \div(\El_{\eta}) = 4\pi \ell.
$$
Indeed, by construction there is no point of $\C$ between $m-\ell$ and $m+\ell$.

Let us observe that  $\Gap_i \preceq \ell$. Thus, using Cauchy-Schwarz's inequality and \eqref{mvgapL2}, we obtain
$$
|\Gap_i|^2 \leq \Gap_i \int_{\partial H_{\ell}} |\El_{\eta}|^2 \preceq \int_{[x_i, x_{i+1}] \times \R}  |\El_{\eta}|^2.
$$

The results follows by summing on $i$, and observing that
$$
\sum_{i = -R/2}^{R/2} \int_{[x_i, x_{i+1}] \times \R}  |\El_{\eta}|^2 \leq \int_{\La_R \times \R} |\El_{\eta}|^2, 
$$
 since by assumption there are at least $R/2$ points on both sides of $\La_R$.
\end{proof}

\subsection{Proof of Lemma \ref{prop:gapsdiff}}
\label{sec:proofgapsdiff}
\begin{proof}
Since $\Pz \neq \Pu$, there exists a continuous function $F : \config \to \R$ with $\|F\|_{\infty} = 1$, and $c > 0$ such that 
\begin{equation}
\label{EPFEQF}
\E_\Pz [ F ] - \E_\Pu [F] = c.
\end{equation}
Furthermore, without loss of generality we may assume that $F$ is \textit{local} in the sense that there exists $N > 0$, such that for any $\C$ in $\config$,
\begin{equation}
\label{Fislocal}
F(\C) = F(\C \cap \La_N).
\end{equation}
Indeed, by dominated convergence, we have
$$
\lim_{R \to \infty} \left( \E_\Pz \left[ F( \cdot \cap \La_R ) \right]  - \E_\Pu \left[ F( \cdot \cap \La_R ) \right] \right) = \E_\Pz [ F ] - \E_\Pu [F].
$$

For any $M > 0$, we define the function $F_M : \config \to \R$ as
$$
F_M(\C) := \frac{1}{2M} \int_{-M}^M F(\C + t) dt.
$$
Since, $\Pz, \Pu$ are stationary, we have, for $M$ arbitrary, in view of \eqref{EPFEQF},
\begin{equation}
\label{EPFM}
 \E_\Pz [ F_M ] - \E_\Pu [F_M] = c > 0.
\end{equation}
Also, since $F$ is local and satisfies \eqref{Fislocal}, we have
\begin{equation}
\label{FMislocal}
F_M(\C) = F_M\left( \C \cap \La_{M+N} \right).
\end{equation}

Recall that, in Definition \ref{def:gaps}, we denote by $x_0(\C)$ the first non-negative point of $\C$. Since $\Pz, \Pu$ are assumed to have finite energy, the discrepancy estimate \eqref{firstoftennotfar} holds and we have
$$
\lim_{T \to \infty} \Pz\left(x_0(\C) > T\right) + \Pu\left(x_0(\C) > T\right)  = 0,
$$
so we may choose $T$ (depending on $c, \Pz, \Pu$) such that
\begin{equation}
\label{conditionT}
\Pz\left(x_0(\C) > T\right) + \Pu\left(x_0(\C) > T\right) \leq \frac{c}{100}.
\end{equation}

Once $T$ is fixed, we choose $M$ such that
\begin{equation}
\label{conditionM}
\frac{T}{M} \leq \frac{c}{100}.
\end{equation}
We may also impose that $M > N$, where $N$ is as in \eqref{Fislocal}, so in view of \eqref{FMislocal} we have for any $\C$ in $\config$,
\begin{equation}
\label{FMislocalbis}
F_M(\C) = F_M\left( \C \cap \La_{2M} \right).
\end{equation}

\begin{claim}
\label{claim:recenter}
If $x_0(\C) \leq T$, and \eqref{conditionM} holds, we have
\begin{equation}
\label{recenter}
|F_M(\C) - F_M\left(\C - x_0(\C)\right)| \leq \frac{c}{100}
\end{equation}
\end{claim}
\begin{proof}[Proof of Claim \ref{claim:recenter}]
We use the definition of $F_M$, the fact that $F$ is bounded by $1$, and the choice of $M$ with respect to $T$ as in \eqref{conditionM}, and compute
\begin{multline*}
|F_M(\C) - F_M(\C - x_0(\C))| \leq \frac{1}{2M} \left( \int_{M - x_0(\C)}^M |F(\C + t)| dt  + \int_{-M - x_0(\C)}^{-M} |F(\C+t)| dt \right) \\
 \leq \frac{2 x_0(\C)}{2M} \|F\|_{\infty} \leq \frac{T}{M} \leq \frac{c}{100}.
\end{multline*}
\end{proof}

Since $\Pz, \Pu$ have intensity $1$, for $r \geq 1$ large enough (depending on $M, c, \Pz, \Pu$), we have
\begin{equation}
\label{choicer}
\Pz\left(|\C \cap \La_{8M}| > r \right) + \Pu\left(|\C \cap \La_{8M}| > r\right) \leq \frac{c}{100}. 
\end{equation}
For $r$ fixed such that \eqref{choicer} holds, we consider the map on $\R^{2r+1}$ defined by
\begin{equation}
\label{def:H}
H\left(a_{-r}, \dots, a_{r}\right) := F_M\left( \delta_0 + \sum_{i=1}^r (\delta_{p_i} + \delta_{p_{-i}}) \right),
\end{equation}
where we let the $p_i$'s be 
\begin{equation}
\label{def:pi}
p_i := \sum_{k = 1}^i a_k \ (i \geq 1), \quad p_{-i} : = \sum_{k=1}^i - a_{-k} \ (i \geq 1).
\end{equation}
In other words $H\left(a_{-r}, \dots, a_{r}\right)$ is obtained by applying $F_M$ to the configuration made of one point at $0$ and $2r$ points located at $p_i$ for $|i| = 1, \dots, r$. The idea is that if $a_{-r}, \dots, a_r$ are the first gaps of $\C$, then this configuration is made of the first points of $\C - x_0(\C)$. Since $F_M$ is bounded by $1$, clearly so is $H$.

\begin{claim}
\label{claim:EPHEQH}
We have
\begin{equation}
\label{EPHEQH}
\E_{\Pz} \left[H\left(\Gap_{-r}(\C), \dots, \Gap_{r}(\C) \right) \right] - \E_{\Pu} \left[H\left(\Gap_{-r}(\C), \dots, \Gap_{r}(\C) \right) \right] > 0.
\end{equation}
\end{claim}
\begin{proof}[Proof of Claim \ref{claim:EPHEQH}]
First of all, since $\Pz, \Pu$ have finite energy, the configurations have almost surely infinitely many points in $\R_+$ and $\R_-$, hence the gaps are almost surely all finite and $H\left(\Gap_{-r}(\C), \dots, \Gap_{r}(\C) \right)$ is well-defined almost surely.

Since $H$ is bounded by $1$, and since \eqref{conditionT}, \eqref{choicer} hold, we have
\begin{equation}
\label{EPHEPHcond}
\left|  \E_{\Pz} \left[ H\left(\Gap_{-r}, \dots, \Gap_{r} \right) \right] - \E_{\Pz} \left[\1_{x_0(\C) \leq T} \1_{|\C \cap \La_{8M}| \leq r} H\left(\Gap_{-r}, \dots, \Gap_{r} \right) \right]  \right| \leq \frac{c}{100} + \frac{c}{100},
\end{equation}
and the same holds for $\Pu$. Knowing that $x_0(\C) \leq T$, we have by \eqref{recenter}
$$
\left|F_M (\C - x_0(\C)) - F_M(\C) \right| \leq \frac{c}{100}, 
$$
and since $F_M$ satisfies \eqref{FMislocalbis}, we have
$$
F_M (\C - x_0(\C)) = F_M \left( \left(\C - x_0(\C)\right) \cap \La_{2M} \right).
$$
Of course, the point configuration $\C - x_0(\C)$ can be written in terms of the gaps of $\C$ as
$$
\C - x_0(\C) = \delta_0 + \sum_{i=1}^{+\infty} \left(\delta_{p_i} + \delta_{p_{-i}} \right), 
$$
where the $p_i$'s are as in \eqref{def:pi}. Knowing that, moreover, $|\C \cap \La_{8M}| \leq r$, we see that 
$$
\left(\C - x_0(\C)  \right) \cap \La_{2M} = \left(\delta_0 + \sum_{i=1}^r \delta_{p_i} + \delta_{p_{-i}} \right) \cap \La_{2M}, 
$$
and thus, using again \eqref{FMislocal},
\begin{multline*}
F_M \left( \left(\C - x_0(\C)\right) \cap \La_{2M} \right) = F_M \left( \left(\delta_0 + \sum_{i=1}^r (\delta_{p_i} +  \delta_{p_{-i}}) \right) \cap \La_{2M} \right) = F_M \left(\delta_0 + \sum_{i=1}^r (\delta_{p_i} +  \delta_{p_{-i}}) \right)
\\
= H\left(\Gap_{-r}, \dots, \Gap_{r}\right).
\end{multline*}
We thus obtain by \eqref{recenter}
$$
\left| \E_{\Pz} \left[\1_{x_0(\C) \leq T} \1_{|\C \cap \La_{8M}| \leq r} H\left(\Gap_{-r}, \dots, \Gap_{r} \right) \right] - \E_{\Pz} \left[ \1_{x_0(\C) \leq T} \1_{|\C \cap \La_{8M}| \leq r} F_M(\C) \right] \right| \leq \frac{c}{100},
$$
which easily yields, in view of \eqref{EPHEPHcond},
\begin{equation}
\label{EPHcondEPFM}
\left| \E_{\Pz} \left[ H\left(\Gap_{-r}, \dots, \Gap_{r} \right) \right] - \E_{\Pz} \left[ F_M(\C) \right] \right| \leq  \frac{5c}{100},
\end{equation}
and the same goes for $\Pu$, which proves the claim.
\end{proof}

Finally, we use a density argument in $L^1(\R_{+}^{2r+1})$ to find a test function on $\R^{2r+1}_{+}$ that satisfies \eqref{EPHEQH} \textit{and} is compactly supported and Lipschitz with respect to the $\| \cdot \|_{1}$ norm. By possibly reducing the lower bound in \eqref{EPHEQH}, we can assume the test function to be $1-$Lipschitz.
\end{proof} 

\subsection{Proof of Proposition \ref{prop:gaps2}}
\label{sec:proofgaps2}
\begin{proof}[Proof of Proposition \ref{prop:gaps2}] 
We recall that $\config(\R)$ denotes the space of point configuration on $\R$.
\subsubsection*{Detecting the local difference}
Let $c > 0, r \geq 1$, and a function $H : \R_{+}^{2r+1} \to \R$ as given by Lemma \ref{prop:gapsdiff}, such that
\begin{equation}
\label{EPH}
\E_{\Pz} \left[ H(\Gap_{-r}, \dots, \Gap_r) \right] - 
\E_{\Pu} \left[ H(\Gap_{-r}, \dots, \Gap_r) \right] \geq c.
\end{equation}
The function $H$ is compactly supported, so let $L$ be such that $H$ is supported in $[0, L]^{2r +1}$. Without loss of generality, we can take $L > 10$. To clarify notation, let us define $\tH(\C)$ for a configuration $\C$ in $\config(\R)$ as 
\begin{equation}
\label{def:tH}
\tH(\C) := H(\Gap_{-r}(\C), \dots, \Gap_{r}(\C)).
\end{equation}
Strictly speaking, $\tH$ is not defined everywhere on $\config(\R)$, but it is well-defined on the set of configurations with at least $r+1$ points on $\R_+$ and $\R_-$, because it ensures that the gaps $\Gap_{-r}(\C), \dots, \Gap_{r}(\C)$ are all finite.

\subsubsection*{A test function to detect the global difference}
For $R > 1$, we define the function $\hHR$ on $\config(\R)$ as
\begin{equation}
\label{def:hHR} \hHR : \C \mapsto \int_{0}^{\Rdix} \tH(\C - t) dt.
\end{equation}
Clearly, $\hHR$ is bounded by $\frac{R}{10}$. Since $\Pz, \Pu$ are stationary, we have of course, using \eqref{EPH},
\begin{equation}
\label{EPQhHR}
\E_{\Pz}[ \hHR ] - \E_{\Pu}[ \hHR] \geq \frac{cR}{10}.
\end{equation}
Strictly speaking, $\hHR$ is not defined everywhere on $\config(\R)$. It is well-defined on the set configurations with at least $r+1$ points on $\left[\frac{R}{10}, + \infty\right)$ and on $\left(-\infty, 0 \right]$, because it guarantees that the gaps $\Gap_{-r}(\C-t), \dots, \Gap_{r}(\C-t)$ are finite for all $t \in \left[0, \frac{R}{10}\right]$.

Now, if $\Q$ is a coupling of $\Pz$ and $\Pu$, we may re-write \eqref{EPQhHR} as
\begin{equation}
\label{usingcoupling1}
\E_{\Q} \left[ \int_{0}^{\Rdix} \tH(\Cz - t) dt - \int_{0}^{\Rdix} \tH(\Cu - t) dt \right] \geq \frac{cR}{10}.
\end{equation}
Since, we work only with a coupling of the restrictions of $\Pz$ and $\Pu$ to $\La_R$, we have to restrict the argument above to an event of high probability ensuring in particular that there are at least $r+1$ points in $[-R,0]$ and $[R/10,R]$, so that $\hHR$ depends only on the configuration in $\La_R$, see \eqref{GapsunderPi} for a precise formulation.

\subsubsection*{Strategy of the proof}
To simplify, let us assume that $\tH(\C)$ depends only on the first gap of $\C$ (the one with index $0$, between $x_0(\C)$ and $x_1(\C)$, see Definition \ref{def:gaps}). Since $H$ is Lipschitz, for any two configurations $\Cz, \Cu$ and any $t$, we thus have
$$
|\tH(\Cz - t) - \tH(\Cu - t)| \leq \left|\Gap_0(\Cz-t) - \Gap_0(\Cu-t)\right|.
$$
As $t$ goes from $0$ to $\frac{R}{10}$, the first gap of $\Cz-t$ will successively correspond to the first gap of $\Cz$, then the second one, etc. up to a gap of order $\approx \frac{R}{10}$ in $\Cz$, and similarly for $\Cu -t$. If the gaps of $\Cz -t$ and $\Cu-t$ were always “aligned”, i.e. for any $t$, the index $k_0$ such that $\Gap_0(\Cz-t) = \Gap_{k_0}(\Cz)$ and the index $k_1$ (defined similarly for $\Cu$) are equal, we would write, using a Fubini-type argument
\begin{equation}
\label{preFubiniA}
\int_0^{\frac{R}{10}} |\tH(\Cz - t) - \tH(\Cu - t)| dt \precapprox \sum_{k=0}^{\frac{R}{10}} |\Gap_{k}(\Cz) - \Gap_{k}(\Cu)|,
\end{equation}
and thus, in view of \eqref{def:hHR} and \eqref{EPQhHR}, we would get a lower bound of order $R$ on a certain “gap difference”. Getting a lower bound of the type
\begin{equation}
\label{nosquareisenough}
\sum_{k=0}^{\frac{R}{10}} |\Gap_{k}(\Cz) - \Gap_{k}(\Cu)| \geq c R
\end{equation}
 would be enough for our purposes.

Compared to this situation, the quantity $\tH(\C)$ depends on more than one gap. However, it only depends on a finite number of gaps (here, at most $2r +1$), and the strategy can be easily adapted. The major complication comes from aligning gaps, we describe it below.

We want to detect a difference between the gaps of $\Cz$ and $\Cu$. In the statement of Proposition \ref{prop:gaps2}, we include the possibility of a fixed shift $S$, but let us take $S = 0$, and let us try to transform \eqref{usingcoupling1} into a bound of the type \eqref{nosquareisenough}. Let us consider a typical\footnote{For technical reasons, when writing the formal proof below, we actually enforce that $x_0(\Cz) = x_0(\Cu) = 0$, which twists the enumeration a little bit, but does not modify the general strategy.} situation where $x_0(\Cz) > 0$ and $x_0(\Cu) > 0$. Taking the time $t = 0$ in the integrals, we can bound
\begin{equation}
\label{arfarf}
| \tH(\Cz - 0) - \tH(\Cu - 0)  | \leq \sum_{i=-r}^{r} |\Gap_i(\Cz - 0) - \Gap_i(\Cu - 0)| = \sum_{i=-r}^{r} |\Gap_i(\Cz) - \Gap_i(\Cu)|.
\end{equation}
The right-hand side of \eqref{arfarf} appears in the sum in \eqref{nosquareisenough}, and the right-hand side of \eqref{arfarf} appears in the integral of \eqref{usingcoupling1}, so \eqref{arfarf} can be used to transform the lower bound in \eqref{usingcoupling1} into a lower bound of the type \eqref{nosquareisenough}.
Now let us increase $t$, we have again, using the assumption on $H$,
$$
| \tH(\Cz - t) - \tH(\Cu - t)  | \leq \sum_{i=-r}^{r} |\Gap_i(\Cz - t) - \Gap_i(\Cu - t)|,
$$
and for $t$ small we still have
$$
\Gap_i(\Cz - t) = \Gap_i(\Cz), \quad \Gap_i(\Cu - t) = \Gap_i(\Cu).
$$
However, these identities cease to hold as soon as we encounter a point of $\Cz$ or $\Cu$, i.e. as soon as $t = T_0 := \min\left(x_0(\Cz), x_0(\Cu)\right)$. Indeed, assuming e.g. that the first point encountered is $x_0(\Cz)$, we have, for $t$ slightly larger than $T_0$, but smaller than $x_0(\Cu)$,
$$
\Gap_i(\Cz - t) = \Gap_{i+1}(\Cz), \quad \text{ but still } \Gap_i(\Cu - t) = \Gap_i(\Cu),
$$
so comparing naively the two integrands with the Lipschitz control of $H$ gives us a lower bound 
\begin{equation}
\label{arfarf2}
| \tH(\Cz - t) - \tH(\Cu - t)  | \leq \sum_{i=-r}^{r} |\Gap_{i+1}(\Cz) - \Gap_i(\Cu )|.
\end{equation}
The left-hand side of \eqref{arfarf2} is still present in the integrals of \eqref{usingcoupling1}, however the right-hand side of \eqref{arfarf2} does \textit{not} appear in the sum \eqref{nosquareisenough}, and thus \eqref{arfarf2} is useless for us. To remedy this misalignment, we need to shift the configuration $\Cu$ by $x_0(\Cu) - x_0(\Cz)$, i.e. to add a quantity $x_0(\Cu) - x_0(\Cz)$ to the “proper time” $t$ of $\Cu$. Indeed, we have, for $t$ slightly larger than $T_0 + x_0(\Cu) - x_0(\Cz) = x_0(\Cu)$,
$$
\Gap_i(\Cu - t) = \Gap_{i+1}(\Cu),
$$
and thus comparing $\Cz - t$ and $\Cu -t - (x_0(\Cu) - x_0(\Cz))$ yields again a summand from \eqref{nosquareisenough}. Each time that we encounter a “$k$-th point”, no matter whether it comes from $\Cz$ or from $\Cu$, we need to shift the “proper time” of the other configuration in order to “align the gaps”. Doing this, we effectively “lose” portions of the interval $[0, \Rdix]$ on which we integrate, which possibly deteriorates the lower bound of \eqref{usingcoupling1}. In fact, it turns out that the loss can be expressed in terms of the “gap difference” itself.

Another technicality occurs when one tries to make the Fubini argument yielding \eqref{preFubiniA} rigorously, and we will need to estimate the time “spent” on each gap $\Gap_i$ for $i \in [0, \frac{R}{10}]$.

\subsubsection*{A good event}
\def \EG{\mathsf{EG}}
Let us introduce the following events 
\begin{align*}
\EG_1 := & \llbr \left|\C_{\left[\frac{R}{10}, R\right]}\right| \geq r + 1 \rrbr \cap \llbr \left| \C_{\left[-R, 0\right]}\right| \geq r + 1 \rrbr \\
\EG_2 := & \llbr x_{\Pos(\C;\frac{R}{10}) + r} \leq \frac{R}{4} \rrbr \\
\EG_3 := &
 \llbr x_{R^{1/2}}(\C) \leq 2 R^{1/2} \rrbr \cap \llbr x_{-R^{1/2}}(\C) \geq -2 R^{1/2} \rrbr \\
 \EG_4 := & \llbr \left|\C_{\left[-\Rdix, \Rdix\right]}\right| \leq \frac{R}{3} \rrbr,
\end{align*}
and $\EveGap_R$ as the intersection
\begin{equation}
\label{EventR}
\EveGap_R :=  \EG_1 \cap \EG_2 \cap \EG_3 \cap \EG_4.
\end{equation}
By construction, if $\C$ is in $\EveGap_R$, we know that:
\begin{itemize}
\item (Event $\EG_1$.) There are at least $r+1$ points in $[\frac{R}{10}, R]$ and in $[-R, 0]$, so $\hHR(\C|_{\La_R})$ is well-defined.
\item (Event $\EG_2$.) The $r$-th gap of $\C - \frac{R}{10}$ (which is the “right-most” gap that we consider when applying $\hHR$ to a configuration $\C$) corresponds to points in $\left[0, \frac{R}{4}\right]$. It implies that $\hHR(\C)$ depends only on $\C|_{\La_{\frac{R}{4}}}$, namely
\begin{equation}
\label{hHRbecomeslocal}
\hHR(\C) = \hHR(\C|_{\La_{\frac{R}{4}}}).
\end{equation}
\item (Event $\EG_3$.) The $R^{1/2}$-th point on each side is at distance at most $2R^{1/2}$ . Since by assumption $S$ is a random variable bounded by $R^{1/2}$, it yields
\begin{equation}
\label{bornexS}
|x_{S}(\C)| \preceq R^{1/2}.
\end{equation}
We may note that we also have the very rough bound
\begin{equation}
\label{bornefirstidiot}
x_0(\C) \preceq R^{1/2}.
\end{equation}
\item (Event $\EG_4$.) There are at most $\frac{R}{3}$ points in $\left[-\Rdix, \Rdix\right]$, and in particular
\begin{equation}
\label{xkRdix}
|x_k(\C)| \leq \Rdix \implies |k| \leq \frac{R}{3}.
\end{equation}
\end{itemize}

Since $r$ is fixed, the discrepancy estimates \eqref{discrbasic} and \eqref{xposugeq2u} guarantee that, 
$$
\lim_{R \to \infty} \Pz \left( \EveGap_R \right) = \lim_{R \to \infty} \Pu \left( \EveGap_R \right) = 1, 
$$
so for $R$ large enough, we have
\begin{equation}
\label{choiceR0}
\Pz \left( \EveGap_{R} \right) \geq 1 - \frac{c}{100}, \quad \Pu \left( \EveGap_R \right) \geq 1 - \frac{c}{100}.
\end{equation}

\subsubsection*{The quantity to compute}
For $R$ large enough, we may write, since $\hHR$ is bounded by $\frac{R}{10}$
\begin{equation}
\label{hHReveGap}
\E_{\Pz}[ \hHR ] - \E_{\Pu}[ \hHR] \leq \E_{\Pz}[ \hHR \1_{\EveGap_R} ] - \E_{\Pu}[ \hHR \1_{\EveGap_R}] + \frac{2c}{100} \times \frac{R}{10}.
\end{equation}
On the other hand, since $\Q$ is a coupling of the restrictions of $\Pz, \Pu$ to $\La_R$, in view of \eqref{hHRbecomeslocal}  we have
$$
\E_{\Pz}[ \hHR \1_{\EveGap_R}] - \E_{\Pu}[ \hHR \1_{\EveGap_R}] = \E_{\Q} \left[ \hHR(\Cz) \1_{\EveGap_R}(\Cz) - \hHR(\Cu)\1_{\EveGap_R}(\Cu) \right].
$$
We now turn to compute the quantity
\begin{multline}
\label{GapsunderPi}
\E_{\Q} \left[ \1_{\EveGap_R}(\Cz) \hHR(\Cz) - \1_{\EveGap_R}(\Cu) \hHR(\Cu) \right] \\
= \E_{\Q} \left[ \1_{\EveGap_R}(\Cz) \int_{0}^{\Rdix} \tH(\Cz - t)  dt - \1_{\EveGap_R}(\Cu) \int_{0}^{\Rdix} \tH(\Cu - t)  dt \right].
\end{multline}
Let $\Cz, \Cu$ be fixed, and assume both belong to $\EveGap_R$. 

\subsubsection*{The initial shift}
\label{sec:initialshift}
Let $v_0, v_1$ be defined as
\begin{equation}
\label{choixv}
v_0 = x_0(\Cz), \quad v_1 = x_{S}(\Cu).
\end{equation}
In view of \eqref{bornexS}, \eqref{bornefirstidiot}, we have
\begin{equation}
\label{bornev}
|v_0| + |v_1| \preceq R^{1/2}.
\end{equation}

By the definition \eqref{enumerationB}, \eqref{GapsFrom0} of the enumeration of points and gaps, the choice of $v_0, v_1$ as in \eqref{choixv} ensures that for all $i$
\begin{equation}
\label{translationGaps}
x_{i}(\Cz - v_0) = x_i(\Cz), \quad x_{i}(\Cu - v_1) = x_{i+S}(\Cu), \quad \Gap_i(\Cz - v_0) = \Gap_{i}(\Cz), \quad \Gap_i(\Cu - v) = \Gap_{i+S}(\Cu).
\end{equation}

\begin{remark}[Notational choice]
In the sequel, we will write, for simplicity, $\Cz = \Cz - v_0$ and $\Cu = \Cu - v_1$. This amounts to doing a translation in each integral of \eqref{GapsunderPi}, and the error is of order $|v_0| + |v_1|$ which is bounded as in \eqref{bornev}, and thus negligible for our purposes (we pursue a lower bound of order $R$). It places us in an “ideal” situation where $S = 0$, and where, at time $t=0$, both configurations have a point at $0$.
\end{remark}

\subsubsection*{The proper times}
\def \tz{t^{\z}}
\def \tu{t^{\u}}
\def \Deltaz{\Delta^{\z}}
\def \Deltau{\Delta^{\u}}
\def \PTz{\mathrm{PT}^{\z}}
\def \PTu{\mathrm{PT}^{\u}}
\begin{itemize}
\item We let $T_0 = 0$, and recall (see previous paragraph) that $x_0(\Cz) = x_0(\Cu) = 0$. We define the functions 
$$
\tz_0(t) = t, \quad \tu_0(t) = t.
$$
\item We let $T_1$ be the first time at which we encounter a new point, more precisely:

\begin{equation}
\label{def:T1}
T_1 := \min \left\lbrace t: \ \ \tz_0(t) - \tz_0(T_0) = \Gap_0(\Cz) \text{ or } \tu_0(t) - \tu_0(T_0) = \Gap_0(\Cu) \right\rbrace.
\end{equation}
\item We then define the functions 
\begin{equation}
\label{def:tz1}
\tz_1(t) = \tz_0(t) + \left(\Gap_0(\Cz) - \left(\tz_0(T_1) - \tz_0(T_0) \right) \right), \quad \tu_1(t) = \tu_0(t) + \left(\Gap_0(\Cu) - \left(\tu_0(T_1) - \tu_0(T_0) \right) \right).
\end{equation}
Of course, we have in fact $\tz_0(T_0) = \tu_0(T_0) = 0$, $T_1 = \min(\Gap_0(\Cz), \Gap_0(\Cu))$, and one of the quantities 
$$
\Gap_0(\Cz) - T_1, \quad \Gap_0(\Cu) - T_1
$$
is equal to zero, while the other one is the “shift” 
$$
\Gap_0(\Cz) - \Gap_0(\Cu), \text{ or } \Gap_0(\Cu) - \Gap_0(\Cz),
$$
that we must apply to the configuration with a larger first gap.
\item Assume that $T_{k}, \tz_k, \tu_k$ have been defined for some $k \geq 1$. We let the time $T_{k+1}$ be given by
\begin{equation}
\label{def:Tk}
T_{k+1} :=  \min \left\lbrace t, \ \tz_k(t) - \tz_k(T_k) = \Gap_{k}(\Cz) \text{ or } \tu_k(t) - \tu_k(T_k) = \Gap_k(\Cu) \right\rbrace,
\end{equation}
and we introduce the functions
\begin{equation}
\label{def:tzk}
\begin{cases}
\tz_{k+1}(t) & :=  \tz_k(t) + \left(\Gap_k(\Cz) - \left(\tz_k(T_{k+1}) - \tz_k(T_k) \right)\right), \\ 
 \tu_{k+1}(t) & := \tu_k(t) + \left(\Gap_k(\Cu) - \left(\tu_k(T_{k+1}) - \tu_k(T_k)\right) \right).
 \end{cases}
\end{equation}
Once again, by definition one of the “shifts” is equal to zero, and the other one is given by
\begin{equation}
\label{shift}
\Gap_k(\Cz) - \Gap_k(\Cu) \text{ or } \Gap_k(\Cu) - \Gap_k(\Cz).
\end{equation}
\item Let $\Kmax$ be defined as
\begin{equation}
\label{def:Kaymax}
\Kmax:= \min \left\lbrace k: T_k \geq \Rdix\right\rbrace
\end{equation}
We will only consider $k \leq \Kmax$, in other words we stop the definition when $T_{k} \geq \Rdix$ and let $\Kmax$ be the number of steps. Since \eqref{xkRdix} holds, we have 
$
\Kmax \leq \frac{R}{3}.
$
\end{itemize}

\begin{claim}
\label{claim:shifts}
For all $k$, we have
\begin{equation}
\label{postkbis}
x_0(\Cz - \tz_k(T_k)) = x_k(\Cz) - \tz_k(T_k) = 0, \quad x_0(\Cu - \tu_k(T_k)) = x_k(\Cu) - \tu_k(T_k) = 0,
\end{equation}
and for $t$ in $(T_k, T_{k+1})$, we have, using the notation of \eqref{def:Pos}, 
\begin{equation}
\label{postk}
\Pos(\Cz, \tz_k(t)) = \Pos(\Cu, \tu_k(t)) = k + 1,
\end{equation}
which in particular implies the following “gap alignment” identity for $t\in(T_k, T_{k+1})$:
\begin{equation}
\label{gapsarealigned}
\Gap_0\left(\Cz - \tz_k(t) \right) = \Gap_{k+1}(\Cz), \quad \Gap_0\left(\Cu - \tu_k(t) \right) = \Gap_{k+1}(\Cu).
\end{equation}
\end{claim}
\begin{proof}
We recall that we enforced $x_0(\Cz)  = x_1(\Cz) = 0$, which is \eqref{postkbis} for $k=0$. It allows us to re-write the definition \eqref{def:T1} as 
$$
T_1 = \min\left(t, \ t = \Gap_0(\Cz) \text{ or } t = \Gap_0(\Cu)  \right),
$$
thus in fact $\Gap_0(\Cz) = x_1(\Cz)$ and $\Gap_0(\Cu) = x_1(\Cu)$, and $T_1$ is \textit{the first positive time at which a point of $\Cz$ or $\Cu$ is encountered}.

For $0 < t < T_1$, the configuration $\Cz -t$ has a “first negative point” given by $x_0(\Cz) -t = -t$ and a “first positive point” given by $x_1(\Cz) -t$, and the same holds for $\Cu - t$. Thus \eqref{postk} holds for $k = 0$.

Assume that \eqref{postkbis} holds for some $k \geq 0$. And without loss of generality, assume  that 
$$
T_{k+1} = \min\left( t, \tz_k(t) - \tz_k(T_k)  = \Gap_k(\Cz) \right).
$$
Since we know, by induction hypothesis, that $\tz_k(T_k) = x_k(\Cz)$, we see that $\tz_{k+1}(T_{k+1})$ must be given by $x_k(\Cz) + \Gap_k(\Cz)$, which is equal to $x_{k+1}(\Cz)$. We obtain also, following the definition,
$$
\tu_{k+1}(T_{k+1}) = \tu_{k}(T_{k+1}) + \left(\Gap_{k}(\Cu) - (\tu_k(T_{k+1}) - \tu_k(T_k) )\right),
$$
and using the induction hypothesis $\tu_k(T_k) = x_k(\Cu)$, we obtain 
$$
\tu_{k+1}(T_{k+1}) = x_{k}(\Cu) + \Gap_{k}(\Cu) = x_{k+1}(\Cu),
$$
which proves \eqref{postkbis} at rank $k+1$. We also deduce that \eqref{postk} holds (at rank $k$). 
The claim is thus proven by induction.
\end{proof}

\begin{claim}[Estimate on the time loss]
\label{claim:timeloss}
For a bounded function $F(t)$, we have
\begin{equation}
\label{timelossA}
\left| \int_{0}^{\frac{R}{10}} F(t) dt - \sum_{k=0}^{\Kmax-1} \int_{T_k}^{T_{k+1}} F(\tz_k(t)) dt \right| \leq  \|F\|_{\infty} \sum_{k=0}^{\Kmax} \left |\Gap_{k}(\Cz) - \Gap_{k}(\Cu)\right|,
\end{equation}
and similarly when replacing $\tz$ by $\tu$.
\end{claim}
\begin{proof}
The intervals $(T_k, T_{k+1})$ are disjoint, and the change of variable $t \mapsto \tz_k(t)$ has speed $1$. The “time loss”
$$
\left[0, \frac{R}{10} \right] \setminus \bigcup_{k=0}^{\Kmax-1} \llbr  \tz_k\left((T_k, T_{k+1})\right) \rrbr
$$
comes from the time shifts \eqref{shift}, and its length is thus bounded by the sum of all possible shifts, which yields the right-hand side of \eqref{timelossA}.
\end{proof}

\subsubsection*{Using the Lipschitz bound on shifted intervals}
\begin{claim}
\label{claim:Lipschitz}
For all $k \leq \Kmax$, for $t \in (T_k, T_{k+1})$, we have
\begin{equation}
\label{LipschitzBound2}
\left| \tH(\Cz - \tz_k(t)) - \tH(\Cu - \tu_k(t)) \right| \preceq L \sum_{i=-r}^r \frac{\left|\Gap_{i+1+k}(\Cz) - \Gap_{i+1+k}(\Cu) \right|}{\Gap_{i+1+k}(\Cz) + \Gap_{i+1+k}(\Cu)}.
\end{equation}
\end{claim}

\begin{proof}
We first use the fact that $H$ is $1$-Lipschitz with respect to the $\| \cdot \|_{1}$ norm on $\R_+^{2r+1}$, and the identity \eqref{postk} to obtain 
\begin{equation}
\label{LipschitzBound1}
\left| \tH(\Cz - \tz_k(t)) - \tH(\Cu - \tu_k(t)) \right| \leq \sum_{i=-r}^r \left|\Gap_{i+k+1}(\Cz) - \Gap_{i+k+1}(\Cu) \right|.
\end{equation}
To improve \eqref{LipschitzBound1} into \eqref{LipschitzBound2}, we
recall that the function $H$ is bounded by $1$ and supported in
$[0,L]^{2r+1}$ for some $L > 10$.

For any $i$ between $-r$ and $r$,
let us distinguish cases:
\begin{itemize}
\item The gaps $\Gap_{i+k+1}(\Cz)$ and $\Gap_{i+k+1}(\Cu)$ are both smaller than $2L$, in which case we certainly have
\begin{equation*}
 \left| \Gap_{i+k+1}(\Cz) - \Gap_{i+k+1}(\Cu) ) \right|
 \leq  4L \frac{\left| \Gap_{i+k+1}(\Cz) - \Gap_{i+k+1}(\Cu) \right|}{ \Gap_{i+k+1}(\Cz) + \Gap_{i+k+1}(\Cu) )}.
\end{equation*}
\item Both are larger than $L$, in which case the left-hand side of \eqref{LipschitzBound2} is zero.
\item One of these quantities is smaller than $L$ (say, without loss of generality, the first one) and the other one is larger than $2L$, in which case the left-hand side of \eqref{LipschitzBound2} is bounded by $1$, and the right-hand side contains the term
$$
4L \frac{\left| \Gap_{i+k+1}(\Cz) - \Gap_{i+k+1}(\Cu) \right|}{  \Gap_{i+k+1}(\Cz) + \Gap_{i+k+1}(\Cu)} \geq 4L \frac{\Gap_{i+k+1}(\Cu) - L}{L + \Gap_{i+k+1}(\Cu)}, 
$$
but clearly $\frac{4L (x-L)}{x+L} \geq 2$ for $x \geq 2L$ and
$L > 10$, so the right-hand side of \eqref{LipschitzBound2} is bounded
below by $2$, and the inequality holds.
\end{itemize}
\end{proof}

\subsubsection*{Shifting integrals and a Fubini argument}
\def \Weight{\mathrm{Weight}}
\begin{claim} 
\label{claim:shiftandFubini}
\begin{equation}
\label{shiftFubini}
\left| \int_{0}^{\Rdix} \tH(\Cz -t) dt - \int_{0}^{\Rdix} \tH(\Cu - t) dt \right|
\leq \sum_{i = -\infty}^{+\infty} L \Weight_i \frac{\left|\Gap_{i}(\Cz) - \Gap_{i}(\Cu) \right|}{\Gap_i(\Cz) + \Gap_i(\Cu)} + \Error,
\end{equation}
where $\Weight_i$ satisfies
\begin{equation}
\label{boundWi}
\Weight_i \preceq \sum_{j = i-r-2}^{i+r} \min\left(\Gap_j(\Cz), \Gap_j(\Cu) \right),
\end{equation}
and the range condition
\begin{equation}
\label{weightzero}
\Weight_i = 0 \text{ if } |i| \geq R/2,
\end{equation}
and $\Error$ is bounded by
\begin{equation}
\label{ErrorFub}
\Error \preceq \sum_{i = 0}^{R/2} \left|\Gap_{i}(\Cz) - \Gap_{i}(\Cu) \right|.  
\end{equation}
\end{claim}
\begin{proof}[Proof of Claim \ref{claim:shiftandFubini}]
We use Claim \ref{claim:timeloss} with $F = \tH(\Cz - \cdot)$ and $\tH(\Cu - \cdot)$, and write
\begin{multline}
\label{PFA}
\int_{0}^{\frac{R}{10}} \tH(\Cz -t) dt  - \int_{0}^{\Rdix} \tH(\Cu - t) dt
\\ 
= \sum_{k=0}^{\Kmax-1} \int_{T_k}^{T_{k+1}} \left( \tH(\Cz - \tz_k(t)) - \tH(\Cu - \tu_k(t)) \right) dt + \Error, 
\end{multline}
with $\Error$ bounded as in \eqref{ErrorFub}.

Next, we use  \eqref{LipschitzBound2} from Claim \ref{claim:Lipschitz},  and write, for every $k$
\begin{equation}
\label{PFB}
\left|\int_{T_k}^{T_{k+1}} \tH(\Cz - \tz_k(t)) - \tH(\Cu - \tu_k(t)) dt\right| \preceq  L (T_{k+1} - T_k) \sum_{i=-r}^r \frac{\left|\Gap_{i+k+1}(\Cz) - \Gap_{i+k+1}(\Cu) \right|}{\Gap_{i+k+1}(\Cz) + \Gap_{i+k+1}(\Cu)}.
\end{equation}
Combining \eqref{PFA}, \eqref{PFB}, a Fubini argument yields \eqref{shiftFubini}, with weights
$$
\Weight_i \preceq \sum_{k} \1_{k \in [i-r-1, i+r-1], 0 \leq k \leq \Kmax} \left(T_{k+1} - T_k\right).
$$
This is a telescopic sum, and we may write
$$
\Weight_i \preceq T_{i+r} - T_{\min\{i-r-1,0\}}.
$$
By definition of the times $T_k$, we see that
\begin{equation}
\label{boundFubini}
T_{i+r} - T_{\min\{i-r-1,0\}} \leq \sum_{j = i-r-1}^{i+r} \min\left(\Gap_j(\Cz), \Gap_j(\Cu) \right),
\end{equation}
which yields \eqref{boundWi}. Moreover, we see that $\Weight_i$ is $0$ if $i+r-1<0$ or $i+r-1>\Kmax$, which in particular (since $\Kmax \leq \frac{R}{3}$, taking $R$ large enough with respect to $r$) implies \eqref{weightzero}.
\end{proof}

\subsubsection*{Final computation, step 1.}
\def \Gaingain{\mathrm{gain}_{R}(\Cz, \Cu)}
Let us introduce the quantity
\begin{equation}
\label{def:GCzCu}
\Gaingain := \sum_{i = -\frac{R}{2}}^{\frac{R}{2}} \frac{\left(\Gap_i(\Cz) - \Gap_{i}(\Cu)\right)^2}{\left(\Gap_i(\Cz)\right)^2 + \left(\Gap_{i}(\Cu)\right)^2},
\end{equation}
and let us recall that the goal of this proof is to obtain a lower bound on the expectation, under the coupling $\Q$, of $\Gaingain$ that would be proportional to $R$. 

\begin{claim}
\label{claim:finalcomputm1}
There exists a constant $\CLr$ depending only on $L, r$ such that
\begin{multline}
\label{finalcomputm1}
\left| \int_{0}^{\Rdix} \tH(\Cz -t) dt - \int_{0}^{\Rdix} \tH(\Cu - t) dt\right| \\
\leq \CLr \left(\sum_{i=-\frac{R}{2}}^{\frac{R}{2}} \Gap_k(\Cz)^2 + \Gap_k(\Cu)^2 \right)^{1/2} \left(\Gaingain \right)^{1/2} .
\end{multline}
\end{claim}
\begin{proof}[Proof of Claim \ref{claim:finalcomputm1}]
We write  the result of Claim \ref{claim:shiftandFubini} as follows:
\begin{multline}
\label{beforefine}
\left| \int_{0}^{\Rdix} \tH(\Cz -t) dt - \int_{0}^{\Rdix} \tH(\Cu - t) dt \right|
\preceq \sum_{i = -R/2}^{R/2} L \left( \sum_{j = i-r-2}^{i+r} \min\left(\Gap_j(\Cz), \Gap_j(\Cu) \right) \right)  \frac{\left|\Gap_{i}(\Cz) - \Gap_{i}(\Cu) \right|}{\Gap_i(\Cz) + \Gap_i(\Cu)} \\ + \sum_{i = 0}^{R/2} \left|\Gap_{i}(\Cz) - \Gap_{i}(\Cu) \right|.
\end{multline}

For the first term in the right-hand side of \eqref{beforefine}, we use Cauchy-Schwarz's inequality and the trivial bound
$$
 \min\left(\Gap_j(\Cz), \Gap_j(\Cu) \right)^2  \leq  \Gap_j(\Cz)^2 + \Gap_j(\Cu)^2, 
$$
while for the last term in the right-hand side, we write 
$$
\left|\Gap_{i}(\Cz) - \Gap_{i}(\Cu) \right| = \frac{\left|\Gap_{i}(\Cz) - \Gap_{i}(\Cu) \right|}{\Gap_{i}(\Cz) + \Gap_{i}(\Cu)} \left(\Gap_{i}(\Cz) + \Gap_{i}(\Cu)\right),
$$
and use Cauchy-Schwarz's inequality. In both cases, we obtain a term bounded by the right-hand side of \eqref{finalcomputm1}, up to a multiplicative constant depending only on $L, r$.
\end{proof}

\subsubsection*{Final computation, step 2.}
Taking the expectation of \eqref{finalcomputm1} under the coupling $\Q$ (to simplify, we forget about the conditioning on the event $\EveGap_R$, which does not affect the estimate see e.g. \eqref{hHReveGap}. We obtain, replacing the left-hand side of \eqref{finalcomputm1} as in \eqref{GapsunderPi},
\begin{equation}
\label{takingcoupling}
\E_{\Pz}[ \hHR ] - \E_{\Pu}[ \hHR] \preceq \E_{\Q} \left[ \left(\sum_{i=-\frac{R}{2}}^{\frac{R}{2}} \Gap_k(\Cz)^2 + \Gap_k(\Cu)^2 \right)^{1/2} \left(\Gaingain \right)^{1/2}  \right].
\end{equation} 
Using Cauchy-Schwarz's inequality, Lemma \ref{lem:gapL2}, \eqref{EnergyStatCase} and the fact that $\Pz, \Pu$ have finite energy, the right-hand side of \eqref{takingcoupling} can be bounded in order to yield
$$
\E_{\Pz}[ \hHR ] - \E_{\Pu}[ \hHR] \preceq R^{1/2} \E_{\Q} \left[ \Gaingain \right]^{\hal},
$$
and if we use \eqref{EPQhHR} to bound below $\E_{\Pz}[ \hHR ] - \E_{\Pu}[ \hHR]$ we deduce that
$$
R \preceq R^{1/2} \E_{\Q} \left[ \Gaingain \right]^{\hal}, 
$$
so $R \preceq \E_{\Q} \left[ \Gaingain \right]$, which concludes the proof.
\end{proof}

\section{The screening procedure}
\label{sec:proofscreening}
We present a sketch of the screening argument as developed in \cite[Section 6]{petrache2017next}. For the notation of that paper, our setting is $d = 1,\, s = 0,\, k=1, \gamma = 0,\, g(x) = - \log |x|,\, c_{d,s} = 2\pi$, and our $s$ is their $\epsilon$. We also use the fact that, in the present case, the background measure has a constant intensity. We do not claim to make any serious improvement on the procedure.
\begin{proof}[Proof of Proposition \ref{prop:screening}]
We recall that $R' = R(1-s)$, that by assumption \eqref{definiM}, \eqref{decrvert} we have
\begin{align}
\label{goodboundary1}
\int_{\{-R', R'\} \times [-R, R]} |\El_{\eta}|^2 & = \Mec \leq M, \\
\label{gooddecay1} \frac{1}{s^4 R} \int_{\La_R \times \R \backslash (-s^2 R, s^2 R)} |E|^2 & = \Hec \leq 1,
\end{align}
and, that by assumption \eqref{nochargeneargoodboundary}, the smeared out charges $\sigma_{x,\eta}$, for $x$ in $\C$, do not intersect the boundary $\La_{R'} \times \R$.

By a mean value argument, there exists $\l$ in $[s^2 R, 2 s^2 R]$ such that
\begin{equation}
\label{goodboundary2}
\int_{\La_R \times \{-\l, \l\} } |E|^2 \leq s^2 \Hec \leq s^2.
\end{equation}

\paragraph{\textbf{Subdividing the domain.}}
As depicted on Figure \ref{fig:screening}, we decompose $\La_R \times [-R,R]$ in three parts:
\begin{equation*}
D_0 = \La_{R'} \times [-\l, \l], \quad D_{\partial} = \left( \La_R \times [-\l, \l] \right) \backslash D_0, \quad D_1 = (\La_R \times [-R, R]) \backslash (D_0 \cup D_{\partial}).
\end{equation*}
\begin{figure}[ht]
\begin{center}
\includegraphics[scale=1]{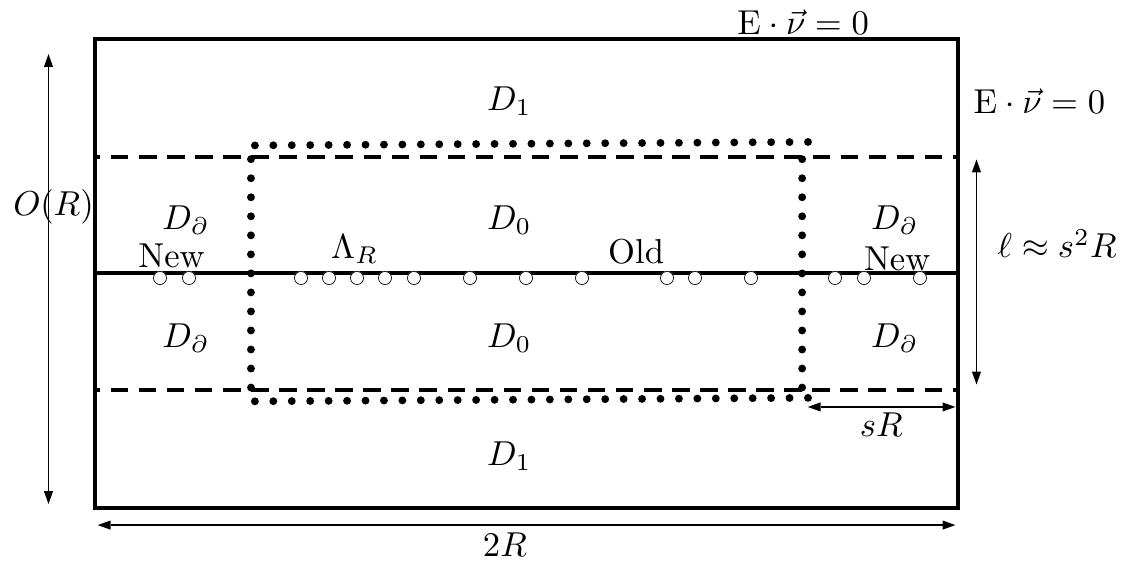}
\caption{Sketch of the situation.}
\label{fig:screening}
\end{center}
\end{figure}

Roughly speaking, here is what the screening procedure entails:
\begin{itemize}
\item The point configuration will be kept in $\Old$ and the existing electric field will be kept in $D_0$.
\item In $D_{\partial}$, we throw away the field and configuration. We will place a correct number of new points and define an electric field whose normal component coincides with the existing one on the vertical dotted lines and reaches $0$ on the vertical full line. There will be, however, a non-vanishing component of the field at the level of the dashed line.
\item In $D_1$, we manipulate the electric field, starting from the horizontal dashed line, in order to reach a zero normal component on the exterior (full line).
\end{itemize}

\begin{definition}[Some quantities.]
We let $\Nint = |\C \cap \La_{R'}| = |\C \cap D_0|$, be the number of points of $\C$ inside $D_0$ - those will not be touched, and we will place $|\La_R| - \Nint$ points in $\New$ in order for the final configuration to have $|\La_R|$ points.

We also define $\UZ$ as the quantity
\begin{equation}
\label{def:UZ}
\UZ := \frac{1}{2 (|\La_R| - |\La_{R'}|)} \int_{[-R',R'] \times \{-\l, \l\}} \El \cdot \vec{\nu}.
\end{equation}
The integral in \eqref{def:UZ} corresponds to the integral of the normal component of the existing field on the horizontal dotted lines in Figure \ref{fig:screening}.
\end{definition}
The screening procedure enforces that the normal component of the constructed field on the dashed line part of the boundary of $D_{\partial}$ exactly compensates $\UZ$, so that the total flux on the boundary of both rectangles that form $D_1$ is $0$.

\begin{claim}[Size of $\UZ$]
We have
\begin{equation}
\label{sizeUZ}
\UZ \preceq R^{-1/2}.
\end{equation}
\end{claim}
\begin{proof}
Since $R' = R(1-s)$ and by definition of $\UZ$, we have
$$
\UZ \preceq \frac{1}{s R} \int_{ [-R',R'] \times \{-\l, \l\} } |E|, 
$$
and combining the Cauchy-Schwarz's inequality with \eqref{goodboundary2} we obtain \eqref{sizeUZ}.
\end{proof}

We split $\La_R \backslash \La_{R'}$ into intervals $H_i$ whose lengths belong to $[\l/2, 2\l]$, and we let $\tH_i = H_i \times [-\l, \l]$. We denote by $\Hleft$ (resp. $\Hright$ the interval exactly to the left (resp. to the right) or $-R'$ (resp. $R'$).

For any interval in this decomposition, we let $m_i$ be such that
\begin{equation}
\label{def:mi}
2\pi (m_i - 1) |H_i| = \int_{ \left(\{-R', R'\} \times [-\l, \l] \right) \cap \partial \tH_i} \El_{\eta} \cdot \vec{\nu} - 2 \UZ |H_i|.
\end{equation}
The first term in the right-hand side of \eqref{def:mi} is only present if $H_i$ is $\Hleft$ or $\Hright$, the second term is always present. 

\begin{claim}[Estimate on $m_i$]
\label{claim:onmi}
We quantify how close to $1$ the number $m_i$  defined in \eqref{def:mi} is.
\begin{itemize}
\item If $H_i$ is not $\Hleft$ or $\Hright$, we have
then 
\begin{equation}
\label{sizemi1}
|m_i - 1| \preceq R^{-1/2}
\end{equation}
\item If $H_i$ is one of the two intervals $\Hleft, \Hright$ that have an intersection with $\partial \La_{R'}$, we have
\begin{equation}
\label{sizemi}
|m_i - 1| \preceq R^{-1/2} + \Mec^{1/2} R^{-1/2} s^{-1}. 
\end{equation}
\end{itemize}
In particular, in both cases, for $R$ large enough, we have
\begin{equation}
\label{mimoins1hal}
|m_i - 1| \preceq \hal.
\end{equation}
\end{claim}
\begin{proof}
In the first case, we have $|m_i - 1| \preceq \UZ$ and we apply \eqref{sizeUZ}. In the second case, we add the contribution of the integral
$$
\frac{1}{|H_i|} \int_{\partial D_0 \cap \partial \tH_i} \El_{\eta} \cdot \vec{\nu} \preceq \l^{-1} \l^{1/2} \left( \int_{\partial D_0} |\El_{\eta}|^2 \right)^{1/2} \preceq (s^2 R)^{-1/2}  \Mec^{1/2},
$$
where we have used the Cauchy-Schwarz inequality and \eqref{goodboundary1}.
\end{proof}

Then each interval $H_i$ is divided into sub-intervals of length $\frac{1}{m_i}$, and in each of these intervals, exactly one point of the new screened configuration is placed. More precisely, a point is placed randomly at distance less than $\frac{\eta}{4}$ from the center of the interval. This randomness is important, because it “creates volume” and yields the description of the “new points” as in \eqref{CscrNew}.

For the end of the screening procedure (constructing the screened electric field and estimating its energy), we refer to the proof of \cite{petrache2017next}. Let us emphasize that, in our setting, some technicalities become irrelevant and that the construction and estimates could be written in a much more concise way. Here is a short sketch thereof:
\begin{enumerate}
\item First, as the $k$-th point $\Xx_k$ is being “placed”, in the sub-interval $I_k$ we define the electric field by solving
$$
- \div\ \El = 2\pi \left( \delta_{\Xx_k}  - dx \right)
$$ 
in $I_k \times [-\l, \l]$, with some boundary condition. These boundary conditions are chosen in a compatible fashion for two neighboring sub-intervals. We also impose a zero boundary condition on the left for the sub-interval that contains $-R$ (and similarly for the rightmost sub-interval), and to match the existing boundary condition given by the pre-existing field $\El$ for the two sub-intervals that share an endpoint with $\Old$. 
Estimating the “energy” created this way is an additional task, but we find that it only yields a small error compared to the total energy.
\item At this stage, the electric field is defined on $\Old \times [-\l, \l]$, the region denoted by $D_0$ in Figure \ref{fig:screening}, by keeping the pre-existing field; and on $\New \times [-\l, \l]$, the region denoted by $D_{\partial}$ in Figure \ref{fig:screening}, by defining it on each sub-interval, as in the previous step. It remains to define it on the region denoted by $D_1$ in Figure \ref{fig:screening}. Of course, we do not place any points here, but we tile this region by small rectangles of side-length $\approx \l$, and on each of these we solve
$$
- \div\ \El  = 0,
$$
with an appropriate choice of (mutually compatible) boundary conditions, which allow us to pass from whatever boundary condition exists at the frontier of $D_{\partial}, D_0$, to the desired Neumann condition $\El \cdot \vec{\nu} = 0$ on the boundary on the big rectangle in Figure \ref{fig:screening}. There again, one must estimate the energy of these “patching” fields. 
\item We then obtain the desired screened electric field.
\end{enumerate}
\end{proof}

\begin{proof}[Proof of Claim \ref{claim:discrenearboundary}]
By construction, for any $H_i$ we place exactly one point at the center of each sub-interval of length $\frac{1}{m_i}$. Since the length of $H_i$ is in $[\l/2, 2\l]$, with $\l \preceq s^2 R$, and since $|m_i - 1| \leq \hal$ (see \eqref{mimoins1hal}) the number of sub-intervals in each interval is $\preceq s^2 R$. In particular
\begin{equation}
\label{pointsinHleft}
\# \text{ points in $\Hleft$ } \preceq s^2R.
\end{equation}

In view of Claim \ref{claim:onmi}, the distance between the position of $k$-th point (starting from the leftmost one) and $-R + k - \hal$ is bounded, up to a universal multiplicative constant, by
\begin{itemize}
\item $k R^{-1/2}$ as long as the point does not belong to $\Hleft$.
\item $k R^{-1/2} + (k-k_0) \Mec^{1/2} R^{-1/2} s^{-1}$ if the point belongs to $\Hleft$,
where $k_0$ is the index of the last point outside $\Hleft$. By \eqref{pointsinHleft}, we may write
$$
k R^{-1/2} + (k-k_0) \Mec^{1/2} R^{-1/2} s^{-1} \preceq \left(1+\Mec^{1/2}\right) s R^{1/2}.
$$
\end{itemize}

We thus have, as claimed in \eqref{xkmoinsk},
$$
\left| \Xx_k - \bXx_k \right| \preceq k R^{-1/2},
$$
as long as $\Xx_k$ is not in $\Hleft$, that is for $k$ such that
$$
s^2 R \preceq |sR - k|.
$$
When $\Xx_k$ is in $\Hleft$, that is for $k$ such that
$$
|sR - k| \preceq s^2 R,
$$
we have, as claimed in \eqref{xkmoinskB},
$$
\left| \Xx_k - \bXx_k \right| \preceq \Mec^{1/2}  s R^{1/2}.
$$

Let us recall that $\kkmax$ is the index of the first point such that $\Xx_k \geq -R + sR$. 
We deduce that $\kkmax$ is equal to $sR$, up to an error of order $\Mec^{1/2} sR^{1/2}$, which yields \eqref{bornekmax}. 

The inequalities \eqref{xkmoinsk}, \eqref{xkmoinskB} can then be converted into discrepancy estimates \eqref{discrNewA} and \eqref{discrNewB}.
\end{proof}

\bibliographystyle{alpha}
\bibliography{uniqueness}

\end{document}